\documentclass[psamsfonts]{amsart}

\usepackage{amssymb,amsfonts}
\usepackage[all,arc]{xy}
\usepackage{enumerate}
\usepackage{mathrsfs}
\usepackage[utf8]{inputenc}
\usepackage{hyperref}
\usepackage{mathtools}
\usepackage{tikz}
\usepackage{enumitem}
\usepackage{amsmath,amssymb,amsthm}
\usepackage{url}

\newtheorem{thm}{Theorem}[section]
\newtheorem{theorem}{Theorem}[section]

\newtheorem{lemma}[thm]{Lemma}

\newtheorem{question}[thm]{Question}
\newtheorem{claim}[thm]{Claim}

\theoremstyle{definition}
\newtheorem{definition}[thm]{Definition}

\theoremstyle{remark}
\newtheorem{remark}[thm]{Remark}

\newtheorem{convention}[thm]{Convention}

\newtheorem{notation}[thm]{Notation}

\makeatletter
\let\c@equation\c@thm
\makeatother
\numberwithin{equation}{section}

\title[Reflection principles, GCH and the uniformization properties]{Reflection principles, GCH and the uniformization properties}

\author{Jing Zhang}

\newcommand{\Addresses}{{
  \bigskip
  \footnotesize

 \textsc{Department of Mathematics,\\ Bar-Ilan University, Ramat-Gan 5290002, Israel}\par\nopagebreak
  \textit{E-mail}: \texttt{jingzhan@alumni.cmu.edu}
}}

\begin{document}

\begin{abstract}
Reflection principles (or dually speaking, compactness principles) often give rise to combinatorial guessing principles. Uniformization properties, on the other hand, are examples of anti-guessing principles. We discuss the tension and the compatibility between reflection principles and uniformization properties at the level of the second uncountable cardinal. 
\end{abstract}

\maketitle
\let\thefootnote\relax\footnotetext{2010 \emph{Mathematics Subject Classification}. Primary: 03E02, 03E35, 03E55. 

Supported by the Foreign Postdoctoral Fellowship Program of the Israel Academy of Sciences and Humanities and by the Israel Science Foundation (grant agreement 2066/18).
}

\section{Introduction, backgrounds and preliminaries}

Reflection principles usually entail certain combinatorial guessing principles. For example, at the level of inaccessible cardinals, $\kappa$ being a measurable cardinal (or just a subtle cardinal) implies $\diamondsuit(\kappa)$ holds, due to Jensen-Kunen \cite{JensenKunen}. If $\kappa$ is supercompact, then there exists a function from $\kappa$ to $V_\kappa$ that exhibits very strong guessing properties, due to Laver \cite{MR0472529}. See \cite{Hamkins:LaverDiamond} for a finer analysis at smaller large cardinals. There are also studies on the compatibility of compactness principles and the failure of guessing principles at inaccessible cardinals, see for example \cite{CummingsWoodin}, \cite{MR3960897}, \cite{MR3914938}, \cite{MR2279655}, \cite{golshani2016weak}, 

At the level of successor cardinals, there are also analogous phenomena. However, to avoid trivialities, we must be careful about the questions asked. Recall by a celebrated result of Shelah \cite{MR2596054}, for any uncountable cardinal $\kappa$, $2^\kappa=\kappa^+$ implies $\diamondsuit(S)$ holds for any stationary $S\subset \{\alpha<\kappa^+: \mathrm{cf}(\alpha)\neq \mathrm{cf}(\kappa)\}$. In the presence of GCH, it still makes sense to ask the validity of diamond at the \emph{critical cofinality}, namely $\diamondsuit(E^{\kappa^+}_{\mathrm{cf}(\kappa)})$. The study now is naturally divided into two cases: successors of a regular uncountable cardinal, or successors of a singular cardinal. For more information regarding the case on the successors of a singular cardinal, see for example \cite{MR2587470}, \cite{MR2869191}, \cite{MR2723781}, \cite{MR2777747}.

We will focus on successors of a regular uncountable cardinal in this paper. For simplicity, we will deal with the second uncountable cardinal $\omega_2$; the generalization to any other successor of a regular uncountable cardinal is straightforward. It is a theorem of Shelah \cite{MR597452} that it is consistent with GCH that $\diamondsuit(E^{\omega_2}_{\omega_1})$ fails. In fact, he obtained a stronger conclusion. To state the result, we need some definitions. Let $S^2_0$ denote $E^{\omega_2}_{\omega}$ and $S^2_1$ denote $E^{\omega_2}_{\omega_1}$ for the rest of this article.

\begin{definition}
A sequence of functions $\bar{\eta}=\langle \eta_\delta: \delta\in S^2_1\rangle$ is called a \emph{ladder system} on $S^2_1$ if for each $\delta\in S^2_1$, $\eta_\delta$ is an unbounded subset of $\delta$ of order type $\omega_1$. A \emph{ladder coloring} on $\bar{\eta}$ is a sequence of functions $\bar{c}=\langle c_\delta\in {}^{\eta_\delta} 2: \delta\in S^2_1 \rangle$. A ladder coloring $\bar{c}$ is constant if  each $c_\delta$ is a constant function.
\end{definition}

\begin{definition}
Given a ladder system $\bar{\eta}$ on $S^2_1$, we say the \emph{2-uniformization property} holds for $\bar{\eta}$ (abbreviated as $\mathrm{Unif}_2(\bar{\eta})$), if for any ladder coloring $\bar{c}$ on $\bar{\eta}$, there exists a \emph{uniformizing function} $h: \omega_2\to 2$ for $\bar{c}$, namely $h$ satisfies that for any $\delta\in S^2_1$, there exists $\gamma<\delta$ such that $h(\gamma')=c_\delta(\gamma')$ for all $\gamma'\in \eta_\delta - \gamma$.
\end{definition}

\begin{remark}
For any cardinal $\lambda$ and any ladder system $\bar{\eta}$ on $S^2_1$, we can define the \emph{$\lambda$-uniformization property} or $\mathrm{Unif}_\lambda(\bar{\eta})$ analogously. 
\end{remark}

It is not hard to see that $\diamondsuit(S^2_1)$ implies $\mathrm{Unif}_2(\bar{\eta})$ fails for any ladder system $\bar{\eta}$ on $S^2_1$. In the model of Shelah \cite{MR597452}, GCH holds and there exists a ladder system $\bar{\eta}$ on $S^2_1$ such that $\mathrm{Unif}_2(\bar{\eta})$ holds. Even better, the ladder system in that model witnessing the uniformization property is somewhat ``large''.

\begin{definition}
Given ladder systems $\bar{\eta}$ and $\bar{\nu}$ on $S^2_1$, we say 
\begin{enumerate}
\item $\bar{\eta}$ is \emph{club} if each $\eta_\delta$ is a closed unbounded subset of $\delta$;
\item  $\bar{\eta}$ is \emph{stationary} if each $\eta_\delta$ is a stationary subset of $\delta$;
\item  $\bar{\eta}$ is \emph{indexed by $T$} (where $T\subset \omega_1$) \emph{with respect to $\bar{\nu}$} if for each $\delta\in S^2_1$, $\eta_\delta= \nu_\delta[T]$, where $\{\nu_\delta(i): i<\omega_1\}$ is the increasing enumeration of $\nu_\delta$ and $\nu_\delta[T]=\{\nu_\delta(i): i\in T\}$. We write $\bar{\eta}=\bar{\nu}\restriction T$;
\item $\bar{\eta}$ is \emph{indexed by $T$} if there exists a club ladder system $\bar{\nu}$ such that $\bar{\eta}$ is indexed by $T$ with respect to $\bar{\nu}$.
\end{enumerate}
\end{definition}

The ladder system $\bar{\eta}$ in Shelah's model witnessing that $\mathrm{Unif}_2(\bar{\eta})$ holds is indexed by a stationary co-stationary subset of $\omega_1$. This in some sense is that best we can do, since by another theorem of Shelah \cite[Page 982, Section 3]{MR1623206}, ZFC proves that $\neg \mathrm{Unif}_2(\bar{\eta})$ for any club ladder system $\bar{\eta}$ on $S^2_1$.

It is well-known that unlike $\omega_1$, there are many compactness principles that can consistently hold at $\omega_2$. Hence, though GCH is consistent with $\neg \diamondsuit(S^2_1)$, GCH along with certain assumptions asserting that $\omega_2$ is ``compact'' may imply $\diamondsuit(S^2_1)$.

\begin{definition}
We let 
\begin{enumerate}
\item $\mathrm{Refl}(S^2_0)$ abbreviate the assertion that \emph{every stationary subset $S\subset S^2_0$ reflects}, namely there exists $\delta\in S^2_1$ such that $S\cap \delta$ is stationary in $\delta$;
\item $\mathrm{Refl}_*(S^2_0)$ abbreviate the assertion that \emph{every stationary subset $S\subset S^2_0$ reflects to a club}, namely there exists $\delta\in S^2_1$ such that $S\cap \delta$ contains a club subset of $\delta$.
\item $\mathrm{Refl}_T(S^2_0)$ (where $T\subset \omega_1$ is stationary) abbreviate the assertion that \emph{every stationary subset $S\subset S^2_0$ reflects with pattern $T$}, namely: for any club ladder system $\langle \nu_\delta: \delta\in S^2_1\rangle$ and any stationary $S\subset S^2_0$, there exists $\delta\in S^2_1$ such that $S\cap \nu_\delta [T]$ is stationary in $\delta$.
\end{enumerate}
\end{definition}

Note that 
\begin{itemize}
\item $\mathrm{Refl}_*(S^2_0) \Rightarrow \mathrm{Refl}_T(S^2_0)$ for all stationary $T\subset \omega_1$ and
\item for stationary sets $T\subset T'\subset \omega_1$, $\mathrm{Refl}_T(S^2_0) \Rightarrow \mathrm{Refl}_{T'}(S^2_0)$ and 
\item  $\mathrm{Refl}(S^2_0) \Leftrightarrow \mathrm{Refl}_{\omega_1}(S^2_0)$.
\end{itemize}

Suppose $\mathrm{Refl}_*(S^2_0)$ holds and $2^{\omega_1}=\omega_2$, then $\diamondsuit(S^2_1)$ must also hold (see \cite{MR924672}). For generalizations and variations, see \cite{fakereflection}, \cite{MR656600}. The following local version of a theorem due to Shelah \cite{ShelahStrange} gives another scenario certifying that $\omega_2$ is ``compact enough'' to imply $\diamondsuit(S^2_1)$. 

\begin{definition}
Let $I$ be a countably complete ideal on $\omega_1$.
We say that
\begin{enumerate}
\item $I$ is precipitious if whenever $U$ is a generic $V$-ultrafilter added after forcing $P(\omega_1)/I$ over $V$, $\mathrm{Ult}(V, U)$ is well-founded,
\item $I$ is saturated if $P(\omega_1)/I$ is $\omega_2$-c.c and
\item $I$ is pre-saturated if $P(\omega_1)/I$ is precipitious and $P(\omega_1)/I$ preserves $\omega_2^V$ as a cardinal.
\end{enumerate} 
 
\end{definition}

\begin{theorem}[Shelah]\label{ShelahStarting}
Suppose $2^{\omega_1}=\omega_2$. Suppose further that there exists a stationary set $T\subset \omega_1$ such that $\mathrm{Refl}_T(S^2_0)$ and $\mathrm{NS}_{\omega_1}\restriction T$ is saturated. Then $\diamondsuit(S^2_1)$ holds.
\end{theorem}
\begin{proof}
By \cite{MR2596054}, $2^{\omega_1}=\omega_2$ implies $\diamondsuit(S^2_0)$ holds. We will show $\diamondsuit^-(S^2_1)$ holds, which is well-known to be equivalent to $\diamondsuit(S^2_1)$ by a result of Kunen.
Let $\langle S_\alpha: \alpha\in S^2_0\rangle$ be a $\diamondsuit(S^2_0)$-sequence. Fix a club ladder system $\bar{\nu}=\langle \nu_\delta \subset S^2_0 : \delta\in S^2_1\rangle$. Fix $\delta\in S^2_1$. Consider $A_\delta=_{def}\{T'\subset T: T'\text{ is stationary and for any }\xi<\xi'\in T', S_{\nu_\delta(\xi)}\sqsubseteq S_{\nu_\delta(\xi')}\}$. For each $T'\in A_\delta$, let $B^\delta_{T'}=\bigcup_{\xi\in T'} S_{\nu_\delta(\xi)}$. Let $A'_\delta$ be a maximal subset of $A_\delta$ satisfying that for any $C, D\in A'_\delta$, $C\cap D$ is non-stationary. As $\mathrm{NS}_{\omega_1}\restriction T$ is saturated, $|A'_\delta|\leq \aleph_1$. Let $\mathcal{C}_\delta=\{B^\delta_{T'}:  T'\in A'_\delta\}$. We claim $\langle \mathcal{C}_\delta: \delta\in S^2_1\rangle$ is a $\diamondsuit^-(S^2_1)$-sequence.

Suppose $X\subset \omega_2$ is given. There exists a stationary $S\subset S^2_0$ such that for all $\alpha\in S$, $X\cap \alpha = S_\alpha$. Find $\delta\in S^2_1$ such that $S\cap \nu_\delta[T]$ is stationary. Therefore, there is a stationary $T'\subset T$ such that $\nu_\delta[T']=S\cap \nu_\delta[T]$ and $T'\in A_\delta$. By the maximality of $A'_\delta$, there exists $T^*\in A'_\delta$ such that $T^*\cap T'$ is stationary. This implies $B^\delta_{T^*}=B^{\delta}_{T'}=X\cap \delta\in \mathcal{C}_\delta$.
\end{proof}

It is now a natural question whether we can weaken the hypothesis, specifically regarding the degree of ``compactness'' of $\omega_2$, in Theorem \ref{ShelahStarting}. Our main result (Theorem \ref{main}) describes a consistent scenario (relative to the existence of large cardinals) that GCH holds, $\omega_2$ is ``compact'' in a sense yet $\diamondsuit(S^2_1)$ fails. This serves as an evidence that Theorem \ref{ShelahStarting} is optimal in a sense. In particular, it is not possible to replace the (local) saturation of $\mathrm{NS}_{\omega_1}$ in the hypothesis of Theorem \ref{ShelahStarting} with the pre-saturation of $\mathrm{NS}_{\omega_1}$.

\begin{definition}
$\omega_2$ is generically supercompact via some countably closed forcing if for any $\lambda$, there exists a countably closed forcing $P$, such that whenever $G\subset P$ is generic over $V$, in $V[G]$, there exists an elementary embedding $j: V\to M$ with critical point $\omega_2$ such that $j(\omega_2)>\lambda$ and $j''\lambda\in M$.
\end{definition}

\begin{remark}
The fact that $\omega_2$ is generically supercompact via some countably closed forcing is a strong reflection principle. A game-theoretic equivalent formulation of this principle, called the \emph{Game Reflection Principle}, was studied in \cite{MR2298486}. In particular, it implies that $\mathrm{Refl}_T(S^2_0)$ holds for all stationary $T\subset \omega_1$ and $\mathrm{NS}_{\omega_1}$ is presaturated. See Section \ref{stationaryladder} for more information.
\end{remark}

\begin{theorem}\label{main}
Relative to the existence of a supercompact cardinal, it is consistent that 
\begin{enumerate}
\item GCH holds,
\item $\omega_2$ is generically supercompact via some countably closed forcing, and
\item there exists a ladder system $\bar{\eta}$ on $S^2_1$ such that $\mathrm{Unif}_2(\bar{\eta})$ holds.
\end{enumerate}
\end{theorem}

\begin{remark}\label{UnifImpliesCH}
One may wonder if CH plays a role here. Indeed, if there exists a ladder system $\bar{\eta}$ on $S^2_1$ such that $\mathrm{Unif}_2(\bar{\eta})$ holds, then CH must hold. To see this, note:
\begin{itemize}
\item $\mathrm{Unif}_{2^\omega}(\bar{\eta})$ must hold. Given $\bar{c}=\langle c_\delta: \eta_\delta\to 2^\omega \rangle$, let $\bar{c}^n$ be a ladder coloring on $\bar{\eta}$ such that $c_\delta^n(\gamma)=c_\delta(\gamma)(n)$. Apply $\mathrm{Unif}_2(\bar{\eta})$ to get $f_n: \omega_2\to 2$ uniformizing $\bar{c}^n$ for each $n\in \omega$. It is easy to check that $f: \omega_2 \to 2^\omega$ defined as $f(\alpha)=\langle f_n(\alpha): n\in \omega\rangle$ uniformizes $\bar{c}$.
\item $\mathrm{Unif}_{\omega_2}(\bar{\eta})$ must fail. This is witnessed by the ladder coloring $\bar{c}$ such that for any $\delta\in S^2_1$, $c_\delta \equiv \delta$, namely $c_\delta$ is the constant function $\delta$ on $\eta_\delta$.
\end{itemize}

However, there are still related and meaningful questions in the absence of CH. See Section \ref{questions} for more information.
\end{remark}

The paper is organized as follows: 
\begin{itemize}
\item Section \ref{wellchosen} describes a ladder system on $S^2_1$ and the uniformization forcing with respect to that ladder system; 
\item Section \ref{non-stationary} gives the details of the proof of Theorem \ref{main};
\item Section \ref{variations} presents two variations of the model in Section \ref{non-stationary}. In the first variation, the reflection principle is weakened and the uniformization property is strengthened. In the second variation, we deal with the \emph{monochromatic uniformization property} on a club ladder system;
\item Section \ref{cheaper} describes a model obtained from the existence of a weakly compact cardinal;
\item Section \ref{questions} concludes with some questions and remarks.
\end{itemize}

\section{Nice ladder systems and the uniformization forcing}\label{wellchosen}
We assume the ground model satisfies GCH and $\kappa$ is a Mahlo cardinal. Recall that for any subset of ordinals $A$, $acc(A)=\{\alpha\in A: \sup A\cap \alpha=\alpha\}$ and $nacc(A)=\{\alpha\in A: \sup A\cap \alpha <\alpha\}$.

\subsection{Nice ladder systems}\label{clubladder}
In this subsection, we will define in $V^{\mathrm{Coll}(\omega_1, <\kappa)}$ a club ladder system $\bar{\nu}=\langle \nu_\delta: \delta\in \kappa\cap \mathrm{cof}(\omega_1)\rangle$. The ladder system on which we will force the uniformization property will be indexed by some $T\subset \omega_1$ with respect to $\bar{\nu}$.

Work in $V$ and define a function $f$ on the strongly inaccessible cardinals below $\kappa$ such that whenever such $\gamma$ is given, $f(\gamma)$ is a $(\mathrm{Coll}(\omega_1,  < \gamma) \times \mathrm{Coll}(\omega_1, \gamma^{+++}))$-name satisfying the following: whenever $g*h\subset \mathrm{Coll}(\omega_1,  < \gamma) \times \mathrm{Coll}(\omega_1, \gamma^{+++})$ is generic over $V$, in $V[g*h]$, fix a bijection $F=\omega_1\to (\gamma^{+})^V$ and also a surjection $l =\bigcup h : \omega_1 \to (\gamma^{+++})^V$. Define a $\subset$-increasing continuous sequence $\langle N_\xi: \xi<\omega_1\rangle$ such that 
\begin{enumerate}
\item $(H(\gamma^{+++}))^{V[g]}\in N_0$,
\item $N_{\xi}\in V[g]$ and $N_\xi \prec (H(\gamma^{+4}), \in , \prec^*)^{V[g]}$ where $\prec^*$ is a well-ordering of $(H(\gamma^{+4}))^{V[g]}$,
\item $N_\xi\cap \gamma\in \gamma$ and $N_\xi$ is of size $\aleph_1$,
\item $N_{\xi+1}$ is countably closed, 
\item $F\restriction \xi+1, l\restriction \xi+1,\langle N_\zeta: \zeta\leq \xi\rangle \in N_{\xi+1}$.
\end{enumerate}

For each $\xi<\omega_1$, let $\delta_\xi= N_\xi\cap \gamma$ and let $\nu_\gamma = \{\delta_\xi: \xi<\omega_1\}$. 
Back in $V$, let $f(\gamma)=\dot{\nu}_\gamma$ where $\dot{\nu}_\gamma$ is a $(\mathrm{Coll}(\omega_1,  < \gamma) \times \mathrm{Coll}(\omega_1, \gamma^{+++}))$-name for the object $\nu_\gamma$ defined above.

Now if $G\subset \mathrm{Coll}(\omega_1, <\kappa)$ is generic over $V$, then let $S\subset \kappa$ consist of ordinals that are strongly inaccessible cardinals in $V$. Since $\mathrm{Coll}(\omega_1, <\kappa)$ is $\kappa$-c.c, both $S$ and $S^c\cap \mathrm{cof}^V(>\omega)$ remain stationary.
It is also well-known that $\diamondsuit(S^c\cap \mathrm{cof}(\omega_1))$ holds in $V[G]$. Fix a $\diamondsuit(S^c\cap \mathrm{cof}(\omega_1))$-sequence $\langle s_\alpha: \alpha\in S^c\cap \mathrm{cof}(\omega_1) \rangle$. Define $\langle \nu_\gamma : \gamma\in \kappa\cap \mathrm{cof}(\omega_1)\rangle$ such that if $\gamma\in S$, $\nu_\gamma= (\dot{\nu}_\gamma)^{G\restriction \gamma \times G(\gamma^{+++})}$ and if $\gamma\in S^c\cap \mathrm{cof}(\omega_1)$, then if $s_\gamma$ is a club subset of $\gamma$, then $\nu_\gamma\subset s_\gamma$ is a club subset of order type $\omega_1$ with $nacc(\nu_\gamma)\subset nacc(s_\gamma)$. Otherwise, $\nu_\gamma$ is any club subset of $\gamma$ of order type $\omega_1$.

%
%
%

\begin{remark}
The ladder system $\bar{\nu}$ defined above is \emph{nice}, in the sense that it satisfies the following strong club guessing property: for any club $D\subset \omega_2$, there exists stationarily many $\gamma\in S^2_1$ such that $\nu_\gamma\subset D\cap \gamma$ and $nacc(\nu_\gamma)\subset nacc(D\cap \gamma)$. This is guaranteed already by $\diamondsuit(S^c\cap \mathrm{cof}(\omega_1))$. The part of the ladder defined on $S$ will be useful in the consistency proof in Section \ref{non-stationary}.
\end{remark}

\subsection{Forcing the uniformization property}\label{UnifForcing}
Work in $V[G]$, where $G\subset \mathrm{Coll}(\omega_1, <\kappa)$ is generic over $V$. Let $T\subset \omega_1$ be an uncountable co-stationary set. Our goal is to force the 2-uniformization property on $\bar{\eta}=\bar{\nu}\restriction T$.

Suppose we are given a ladder coloring $\bar{c}=\langle c_\delta\in ^{\eta_\delta}2: \delta\in \omega_2\cap\mathrm{cof}(\omega_1)\rangle$. The \emph{2-uniformization forcing} $P_{\bar{c}}$ consists of $f: \alpha\to 2$ for $\alpha<\omega_2$ such that for any $\delta\leq \alpha$ and $\delta\in \mathrm{cof}(\omega_1)$, $f\restriction \eta_\delta=^* c_\delta$. The order relation is extension. Note that $P_{\bar{c}}$ satisfies that any countable decreasing sequence has a greatest lower bound.

The final forcing is a $<\kappa$-support iteration of 2-uniformization forcings of length $\kappa^+$, namely $\langle P_\gamma, \dot{Q}_\beta: \gamma\leq \kappa^+, \beta<\kappa^+\rangle$ such that $\Vdash_{P_\gamma} \dot{Q}_\gamma = P_{\dot{\bar{c}}}$ for some ladder coloring $\dot{\bar{c}}$ on $\bar{\eta}$.

\begin{definition}
Let $\alpha\leq \kappa^+$ be given. We say $r\in P_\alpha$ is \emph{flat} if there exists $ \beta<\kappa$ such that for any $\gamma\in support(r)$, there is $ f\in V\cap 2^{\beta}$ with $r\restriction \gamma\Vdash_{P_\alpha} r(\alpha)=\check{f}$.
\end{definition}

\begin{notation}
Let $\alpha'\leq \alpha\leq \kappa^+$ be given. Suppose $p\in P_\alpha$ and $q\in P_{\alpha'}$ are such that $q\leq_{P_{\alpha'}} p\restriction \alpha'$, then $q\wedge p$ denote the condition $t\in P_\alpha$ where $$t(\gamma) =\begin{cases}
q(\gamma), & \gamma <\alpha'\\
p(\gamma), & \text{otherwise}
\end{cases}.
$$
\end{notation}

The following facts due to Shelah will be crucial to conclude that the iteration forces $\mathrm{Unif}_2(\bar{\eta})$:

\begin{theorem}[Shelah, \cite{MR597452}]\label{Shelah}
Suppose $T\subset \omega_1$ is an uncountable co-stationary set, $\diamondsuit(T^c)$ holds, and suppose $\bar{\eta}=\bar{\nu}\restriction T$, where $\bar{\nu}$ is a nice club ladder system on $S^2_1$. Let $\kappa=\omega_2$ and $\langle P_\gamma, \dot{Q}_\beta: \gamma\leq \kappa^+, \beta<\kappa^+\rangle$ be a $<\kappa$-support iteration of 2-uniformization forcings with respect to $\bar{\eta}$.
Then the following hold: 
\begin{enumerate}
\item $\{r\in P_{\kappa^+}: r \text{ is flat}\}$ is a dense subset of $P_{\kappa^+}$,
\item $P_{\kappa^+}$ is $\omega_1$-distributive.
\end{enumerate}
\end{theorem}

In particular, under the hypothesis of Theorem \ref{Shelah}, $P_{\kappa^+}$ preserves all cardinals. Furthermore, the fact that $P_{\kappa^+}$ is $\kappa^+$-c.c, which is a consequence of \emph{(2)}, makes it possible (with the standard book-keeping methods) that all ladder colorings which appear in the final model are considered in some intermediate model. Therefore, $P_{\kappa^+}$ forces $\mathrm{Unif}_2(\bar{\eta})$.

To give the reader a sense of how Theorem \ref{Shelah} is proved, we include here a proof sketch for a special case, i.e. the 2-step iteration. The proof presented here is not the simplest one can find for this special case, but it contains the key ideas to prove the general case, which will also be used in the proof of Theorem \ref{main}.

\begin{proof}[Proof sketch of Theorem \ref{Shelah}]
Fix a nice club ladder system $\bar{\nu}$, an uncountable co-stationary subset $T\subset \omega_1$ and $\bar{\eta}=\bar{\nu}\restriction T$.
Let $\bar{c}^0$ be a ladder coloring on $\bar{\eta}$ and $\dot{\bar{c}}^1$ be a $P_{\bar{c}^0}$-name for a ladder coloring on $\bar{\eta}$. The forcing we are dealing with is $P_{\bar{c}^0} * P_{\dot{\bar{c}}^1}$. First, we show that $P_{\bar{c}^0}$ is $\omega_1$-distributive. Let $\dot{f}: \omega_1 \to \mathrm{Ord}$ be a $P_{\bar{c}^0}$-name and $p\in P_{\bar{c}^0}$. Find an $\subseteq$-increasing continuous sequence $\langle N_\xi: \xi<\omega_1\rangle$ such that for some sufficiently large regular $\theta$,
\begin{itemize}
\item $N_\xi \prec (H(\theta), \in , \triangleleft)$ where $\triangleleft$ is a well ordering of $H(\theta)$,
\item $\langle N_{\zeta}: \zeta\leq \xi\rangle\in N_{\xi+1}$ and
\item $|N_\xi|=\aleph_1$ and $N_{\xi+1}$ is countably closed.
\end{itemize}
Let $\delta_\xi= N_\xi\cap \omega_2$. Let $\delta=\sup_{\xi<\omega_1} \delta_\xi$. We then recursively define $\langle p_\xi: \xi<\omega_1\rangle$ such that 
\begin{itemize}
\item $p_0=p$, 
\item $p_{\xi+1}$ decides $\dot{f}\restriction \xi$,
\item $p_{\xi+1}\in N_{\xi+1}$ and 
\item $p_{\xi}$ \emph{obeys} $c_\delta^0$, in the sense that $p_{\xi} \supset c_\delta^0 \restriction \delta_{\xi}$.
\end{itemize}
At limit stages, we just take the inverse limit. Suppose we have defined $p_\xi$ where $\xi<\omega_1$. By the hypothesis, we know that $\mathrm{dom}(p_\xi)\leq \delta_{\xi}$. Therefore, $c_{\delta}^0\restriction [\delta_\xi, \delta_{\xi+1}) \cap \eta_\delta\in N_{\xi+1}$ by the fact that $N_{\xi+1}$  is countably closed and $\eta_\delta$ has order type $\omega_1$. In $N_{\xi+1}$, we can first extend $p_\xi$ to obey $c_\delta^0\restriction [\delta_\xi, \delta_{\xi+1})$, and then extend again to decide $\dot{f}\restriction \xi$. 

Let $p_\infty=\bigcup_{\xi<\omega_1} p_\xi$. It is immediate from the construction that $p_\infty\in P_{\bar{c}^0}$ which forces $\dot{f}$ to be in $V$.

Notice that the construction above works even when $\bar{\eta}$ is a club ladder system on $S^2_1$. Recall a theorem of Shelah saying $\mathrm{Unif}_2(\bar{\eta})$ must fail for any club ladder system $\bar{\eta}$. We must make use of the fact that $\bar{\eta}=\bar{\nu}\restriction T$ and $\diamondsuit(T^c)$ to deal with iterations.

To deal with $P_{\bar{c}^0} * P_{\dot{\bar{c}}^1}$, it is no longer the case that we know what $\dot{\bar{c}}^1$ is in the ground model, since it very much depends on the generic for $P_{\bar{c}^0}$. The \emph{key point} now is instead of building a linear sequence (namely, $\langle p_\xi: \xi<\omega_1\rangle$) as in the previous case, we need to build a tree anticipating all possibilities of $\dot{\bar{c}}^1$. During the process, $\diamondsuit(T^c)$ is used to guess the ``correct'' branch.

Let $\dot{f}: \omega_1\to \mathrm{Ord}$ be a given $P_{\bar{c}^0}*P_{\dot{\bar{c}}^1}$-name and $(p,\dot{q})\in P_{\bar{c}^0}*P_{\dot{\bar{c}}^1}$. We will show there exists $(p', \dot{q}')\leq (p,\dot{q})$ forcing $\dot{f}=\check{g}$ for some $g$. Note that it is immediate from the $\omega_1$-distributivity of $P_{\bar{c}^0}$ that the collection of flat conditions in $P_{\bar{c}^0}*P_{\dot{\bar{c}}^1}$ is dense. 

Making use of the niceness of $\bar{\nu}$, we are able to find $\langle N_\xi: \xi<\omega_1\rangle$ such that for some sufficiently large regular $\theta$, 
\begin{itemize}
\item $\dot{f}, (p,\dot{q})\in N_0$,
\item $N_\xi\prec (H(\theta), \in , \triangleleft)$ where $\triangleleft$ is a well ordering of $H(\theta)$, 
\item $\langle N_\zeta: \zeta\leq \xi\rangle\in N_{\xi+1}$,
\item $|N_\xi|=\aleph_1$ and $N_{\xi+1}$ is countably closed and 
\item letting $\delta_\xi=N_\xi\cap \omega_2$ for each $\xi<\omega_1$ and $\delta=\sup_{\xi<\omega_1} \delta_\xi$, we have $\nu_\delta=\{\delta_\xi: \xi<\omega_1\}$.
\end{itemize}
As a result, $\eta_\delta= \{\nu_\delta(\xi): \xi\in T\}=\{\delta_\xi: \xi\in T\}$. For each $\xi<\omega_1$, let $T_\xi$ consist of all possible $t: \eta_\delta\cap \delta_\xi \to 2$ and notice that $T_\xi\in N_{\xi+1}$ by the countable closure of $N_{\xi+1}$. By copying a $\diamondsuit(T^c)$-sequence to the club $\nu_\delta$, we may assume there exists $\langle t_\xi\in T_\xi: \xi\in T^c\rangle$ such that for any $s: \eta_\delta \to 2$, there exists stationarily many $\xi\in T^c$ such that $s\restriction \eta_\delta\cap \delta_\xi = t_\xi$.

Define $\{q^\xi_t: \xi\in \omega_1, t\in T_\xi\}$ such that:
\begin{itemize}
\item $q^\xi_t\leq (p,\dot{q})$,
\item $\langle q^\xi_t: t\in T_\xi\rangle \in N_{\xi+1}$,
\item $\mathrm{dom}(q^\xi_t(0))\leq \delta_\xi$ and $\Vdash_{P_{\bar{c}^0}} \mathrm{dom}(q^\xi_t(1))\leq \delta_\xi$,
\item $q^\xi_t(0)$ obeys $c^0_\delta$,
\item $q^\xi_t$ obeys $t$, namely $\Vdash_{P_{\bar{c}^0}} q^\xi_t(1)\supset t$,
\item for $t, t'\in T_\xi$, $q^\xi_t(0)=q^\xi_{t'}(0)$,
\item for any $\xi<\xi'$, $t\in T_\xi$ and $t'\in T_{\xi'}$, if $t\sqsubset t'$, then $q^{\xi'}_{t'}\leq q^\xi_t$, 
\item if $\xi\in T^c$, $q^{\xi+1}_{t_\xi}$ decides $\dot{f}\restriction \xi$ (note that if $\xi\in T^c$, $T_\xi=T_{\xi+1}$ so $t_\xi\in T_{\xi+1}$).
\end{itemize}
Suppose for a moment that this construction is successful. We let $p^*=\bigcup_{t\in T_\xi} q^{\xi}_t(0)$. Define a $P_{\bar{c}^0}$-name $\dot{q}^*$ such that $p^*\Vdash_{P_{\bar{c}^0}} \dot{q}^*=\bigcup_{t\in T_\xi, t\subset c^1_\delta} q^{\xi}_t(1)$. Extend $p^*$ further to $p'$ such that $p'$ decides $\dot{c}^1_\delta$ to be $d$ and $\dot{q}^*$ to be $q'$, as $P_{\bar{c}^0}$ is $\omega_1$-distributive. As $\langle t_\xi: \xi<\omega_1\rangle$ satisfies the guessing property, we know there are stationarily many $\xi\in T^c$ such that $t_\xi=d\restriction \eta_\delta\cap \delta_\xi$. Hence, for any such $\xi$, $(p',q')\leq q^{\xi+1}_{t_\xi}$. Therefore, $(p',q')$ decides $\dot{f}$.

Finally, we demonstrate how to construct such $\{q^\xi_t: t\in  T_\xi\}$ by recursion on $\xi<\omega_1$. If $\xi\in acc(\omega_1)$, then $q^{\xi}_t$ is defined to be the greatest lower bound of $\langle q^{\zeta}_{t\restriction \delta_\zeta}: \zeta<\xi\rangle$. If $\xi \in T$, then $q^{\xi+1}_t$ is defined to be the canonical extension of $q^\xi_{t\restriction \delta_\xi}$ obeying $t$. More precisely, we let $a$ be the $\triangleleft$-least extension of $q^\xi_{t\restriction \delta_\xi}(0)$ obeying $c^0_\delta$. The reason why this is possible is because by the induction hypothesis, $\mathrm{dom}(q^\xi_{t\restriction \delta_\xi}(0))\leq \delta_\xi$ and $q^\xi_{t\restriction \delta_\xi}(0)$ obeys $c^0_\delta\restriction \delta_\xi$. Similarly, we can find the $\triangleleft$-least $P_{\bar{c}^0}$-name $\dot{b}$ such that $\Vdash_{P_{\bar{c}^0}} \dot{b}\supset q^{\xi}_{t\restriction \delta_\xi}(1), t \restriction  \delta_{\xi+1}$. Define $q^{\xi+1}_t$ to be $(a, \dot{b})$. Notice that $q^{\xi+1}_t\in N_{\xi+1}$.

 If $\xi\in T^c$, which implies $T_{\xi+1}=T_\xi$, then we extend $q^{\xi}_{t_\xi}$ to $q'=_{def} q^{\xi+1}_{t_\xi}$ in $N_{\xi+1}$ such that $q'$ decides $\dot{f}\restriction \xi$. For each $t'\in T_{\xi+1}-\{t_\xi\}$, define $q^{\xi+1}_{t'}=(q'(0), q^\xi_{t'}(1))$. Note that $q^{\xi+1}_{t'}\leq q^\xi_{t'}$ since by the induction hypothesis, $q^{\xi}_{t'}(0)=q^{\xi}_{t_\xi}(0)$.

It is now clear that $\{q^\xi_t: \xi\in \omega_1, t\in T_\xi \}$ constructed above satisfies the requirements.

Let us say a few words about the limit case. Instead of guessing a single ladder coloring as in the two-step iteration, we need to guess many ladder colorings. Since the size of the support is at most $\aleph_1$, we can still use the $\diamondsuit(T^c)$-sequence to guess all these colorings simultaneously. Then we more or less proceed as in the two-step iteration, dealing with coordinates in the support simultaneously. More details will be spelled out in the next section. \end{proof}

\section{A proof of Theorem \ref{main}}\label{non-stationary}
We start with the ground model $V$ satisfying GCH and $\diamondsuit(\omega_1)$.
Let $\kappa$ be a supercompact cardinal. Let $\mathbb{P}=\mathrm{Coll}(\omega_1, <\kappa)$ and in $V^\mathbb{P}$, we find a nice club ladder system $\bar{\nu}$ as in Section \ref{wellchosen}. Let $T=nacc(\omega_1)$ and $\bar{\eta}=\bar{\nu}\restriction T$. Then in $V^\mathbb{P}$, we can define $\langle P_\gamma, \dot{Q}_\beta: \gamma\leq \kappa^+, \beta<\kappa^+\rangle$  to force $\mathrm{Unif}_2(\bar{\eta})$.

Let $Q=\mathbb{P} * P_{\kappa^+}$. We show that the following hold in $V^Q$: 
\begin{enumerate}
\item GCH,
\item $\mathrm{Unif}_2(\bar{\eta})$,
\item $\omega_2$ is generically supercompact via some countably closed forcing.
\end{enumerate}

The last item is the only thing to verify. Given $(p,\dot{q})\in \mathbb{P}*P_{\kappa^+}$ and $\theta > 2^{2^{\kappa^+}}$, we will find a generic filter $G*H$ containing $(p,\dot{q})$ such that in a further countably closed forcing extension, there exists an elementary embedding $j: V[G*H]\to M$ with critical point $\kappa=(\omega_2)^{V[G*H]}$, $j(\kappa)>\theta$ and $j''\theta\in M$. The proof will build on the ideas from \cite{MR597452} and \cite{MR683153}.

Fix some $j: V\to M$ witnessing that $\kappa$ is $\theta$-supercompact in $V$. Let $G\subset \mathbb{P}$ be generic over $V$ containing $p$. 
Let $G^*\subset \mathrm{Coll}(\omega_1, [\kappa, <j(\kappa)))$ be generic over $V[G]$. Then we can lift $j$ to $j^*: V[G]\to M[G*G^*]$. 
Our goal is to build an $H\subset P_{\kappa^+}$ in $V[G*G^*]$ containing $q=_{\mathrm{def}}(\dot{q})^G$ which is generic over $V[G]$, and in $V[G*H]$, $\mathrm{Coll}(\omega_1, [\kappa,<j(\kappa)))/H$ is countably closed. Furthermore, we need to make sure there exists a \emph{master condition} $l\in j^*(P_{\kappa^+})$ extending ${j^*}''H$. Suppose this can be done, then we can finish the proof as follows: force below $l$ to get a generic $h\subset j^*(P_{\kappa^+})$ over $V[G*G^*]$ (notice that $j^*(P_{\kappa^+})$ is countably closed in $V[G*G^*]$). Then we can further lift $j^*$ to $j^+: V[G*H]\to M[G*G^**h]$. Hence, by going to a countably closed forcing extension over $V[G*H]$, we can find an elementary embedding $j^+$ with domain $V[G*H]$ into some transitive class $N$ with critical point $\kappa = (\omega_2)^{V[G*H]}$, $j^+(\kappa)>\theta$ and ${j^+}'' \theta\in N$. We are done. 

For the rest of the proof, we will demonstrate how to find such $H$ in $V[G*G^*]$ satisfying the aforementioned requirements.

By the definition of $\bar{\nu}$, we know that $j^*(\bar{\nu})(\kappa)$ is defined as follows: there exists a surjection $h=\bigcup G^*(\kappa^{+++}): \omega_1\to (\kappa^{+++})^{V[G]}$ and a bijection $F: \omega_1\leftrightarrow (\kappa^+)^{V[G]}$. In addition, there exists a $\subset$-increasing continuous sequence $\bar{N}=\langle N_\xi: \xi<\omega_1\rangle$ such that 
\begin{enumerate}
\item $(H(\kappa^{+++}))^{V[G]}\in N_0$,
\item $N_{\xi}\in V[G]$ and $N_\xi \prec (H(\kappa^{+4}), \in , \prec^*)^{V[G]}$ where $\prec^*$ is a well-ordering of $(H(\kappa^{+4}))^{V[G]}$,
\item $N_\xi\cap \kappa\in \kappa$ and $N_\xi$ is of size $\aleph_1$,
\item $N_{\xi+1}$ is countably closed, 
\item $h\restriction \xi+1, F\restriction \xi+1,\langle N_\zeta: \zeta\leq \xi\rangle \in N_{\xi+1}$.
\end{enumerate}

Let $\delta_\xi=_{def} N_\xi\cap \kappa$ for each $\xi<\omega_1$. We have $j^*(\bar{\nu})(\kappa)=\{\delta_\xi: \xi<\omega_1\}$, in particular, $\sup_{\xi<\omega_1} \delta_\xi = \kappa$.

\subsection{Constructing the tree of generics}\label{tree}

Let $d: (\kappa^{++})^{V[G]} \leftrightarrow P_{\kappa^+}\cup \{D\subset P_{\kappa^+}: D \text{ is dense open}\}$ be a bijection living in $(H(\kappa^{+++}))^{V[G]}$. By the definition of $\bar{N}$, we know a tail of the models must contain $d$ and $P_{\kappa^+} $. Let $\Upsilon=\{\delta_{\xi+1}: \xi<\omega_1\}$ and note that $\Upsilon=^* j^*(\bar{\eta})(\kappa)$. To simplify the presentation, we may assume $N_0$ contains $d$ and $P_{\kappa^+}$ hence $\Upsilon= j^*(\bar{\eta})(\kappa)$. 

Let $F_1: \omega_1\to P_{\kappa^+} \cup \{D\subset P_{\kappa^+}: D \text{ is dense open}\}$ be defined such that $F_1(\xi)=d(\zeta)$ whenever $h(\xi)=\zeta$ for $\zeta< (\kappa^{++})^{V[G]}$ and $1_{P_{\kappa^+}}$ otherwise.

Define $T_\xi$ for $\xi<\omega_1$ as follows: 
$$T_\xi=\{t: dom(t)=F''\xi, \forall \zeta<\xi, t(F(\zeta)): \Upsilon\cap \delta_\xi\to 2\}.$$

Notice that $T_\xi\in N_{\xi+1}$ hence $T_\xi\subset N_{\xi+1}$. Intuitively speaking, $T_\xi$ collects all possible countable initial segments of the ladder colorings at the ``$\kappa$-th coordinate'' in $j^*(P_{\kappa^+})$. 

For each $t\in T_\xi$, we note that there are two ways of taking restrictions. One way is the usual function restriction denoted as $t\restriction \beta$ for $\beta \in  (\kappa^+)^{V[G]}$. The other way is the restriction to the previously defined tree. More precisely, given $\xi'<\xi<\omega_1$, $t\downarrow \xi'$ is a function in $T_{\xi'}$ such that for each $\zeta<\xi'$, $(t\downarrow \xi')(F(\zeta))=t(F(\zeta))\restriction \delta_{\xi'}\cap \Upsilon$.

It is a fact that any $\diamondsuit(\omega_1)$-sequence remains a $\diamondsuit(\omega_1)$-sequence in any countably closed forcing extension. By copying a $\diamondsuit(\omega_1)$-sequence in $V[G]$ to the club $\{\delta_\xi: \xi<\omega_1\}$, we may assume there exists a sequence in $V[G*G^*]$, $\langle t_\xi: \xi<\omega_1\rangle$ with $t_\xi\in T_\xi$, such that in $V[G*G^*]$, for any function $t$ with domain $(\kappa^+)^V=F''\omega_1$, satisfying that for any $\xi<\omega_1$, $t(F(\xi)): \Upsilon\cap \kappa\to 2$, there exist stationarily many $\xi<\omega_1$ such that $t\downarrow \xi = t_\xi$. 

We recursively define $Q^\xi=\{q^\xi_t: t\in T_\xi\}$ satisfying the following: 
\begin{enumerate}
\item $Q^\xi\in N_{\xi+1}$,
\item each $q^\xi_t$ is of \emph{height} $\leq \delta_\xi$, namely, for each $\beta<(\kappa^+)^{V[G]}$, $\Vdash_{P_\beta} dom(q^\xi_t(\beta))\leq \delta_\xi$, 
\item $q^\xi_t$ obeys $t$, namely for any $\beta\in dom(t)$, $\Vdash_{P_\beta} q^\xi_t(\beta)\supset t(\beta)$,
\item\label{four} for any $t_0, t_1\in T_\xi$, if $t_0\restriction \beta = t_1\restriction \beta$, then $q^\xi_{t_1}\restriction \beta = q^\xi_{t_0}\restriction \beta$, 
\item for any $\xi_0<\xi_1$, $t_0\in T_{\xi_0}$ and $t_1\in T_{\xi_1}$, if $t_0\subset t_1$, then $q_{t_1}^{\xi_1}\leq q_{t_0}^{\xi_0}$, 
\item\label{six} for each limit ordinal $\xi<\omega_1$, $q^{\xi+1}_{t_\xi}$ (note that $t_\xi\in T_\xi = T_{\xi+1}$) belongs to $$\bigcap_{\gamma\leq \xi, F_1(\gamma) \text{ is a dense open subset of }P_{\kappa^+}} F_1(\gamma).$$ Furthermore, for all $t\in T_{\xi+1}$, if $F_1(\xi)\in P_{\kappa^+}$ and extends $q^{\xi}_{t}$, then $q^{\xi+1}_{t} $ extends $F_1(\xi)$.
\end{enumerate}

We now proceed to the recursive construction. Let $q_\emptyset^0=q$. At stage $\beta<\omega_1$ when $\beta$ is a limit ordinal, for $t\in T_\beta$, define $q^{\beta}_t$ to be the greatest lower bound of $\langle q^{\beta'}_{t\downarrow \beta'}: \beta'<\beta \rangle$.
At successor stage $\beta+1$, if $\beta$ is a successor ordinal, then 
for a given $t\in T_{\beta+1}$, we define $q_t^{\beta+1}$ to be the $\prec^*$-least extension of $q_{t \downarrow \beta}^\beta$ obeying $t$. More specifically, for each $\alpha\in dom(t)$, let $\sigma_\alpha=\sigma_\alpha^t$ be the $\prec^*$-least $P_\alpha$-name such that $\Vdash_{P_\alpha} q_{t \downarrow \beta}^\beta(\alpha), t(\alpha) \subset \sigma_\alpha$. The reason why such an element exists is that by the construction, $\Vdash_{P_\alpha} dom(q_{t \downarrow \beta}^\beta(\alpha))\leq \delta_\beta$ and $q^\beta_{t\downarrow \beta} (\alpha)$ extends $t(\alpha)\restriction \delta_\beta$. We define $q^{\beta+1}_t$ such that $q^{\beta+1}_t(\alpha')=\begin{cases}
q^{\beta}_{t\downarrow \beta} (\alpha')  & \text{if }\alpha'\not\in dom(t)\\
\sigma_{\alpha'}  & \text{if }\alpha'\in dom(t)
\end{cases}.$
Notice that the construction above happens in $N_{\beta+1}$. 
To see that the inductive requirements are maintained, note the following:
\begin{itemize}
\item the construction happens in $N_{\beta+1}$, hence $q^{\beta+1}_t$ is of height $\leq \delta_{\beta+1}$,
\item $q^{\beta+1}_t$ obeys $t$ by design,
\item for any $t_0, t_1\in T_{\beta+1}$, if $t_0\restriction \nu= t_1 \restriction \nu$, then by the inductive requirement at $\beta$, we know $q^{\beta}_{t_0\downarrow \beta} \restriction \nu=q^{\beta}_{t_1\downarrow \beta} \restriction \nu$. If $\alpha\in \nu \cap dom(t_0)=\nu \cap dom(t_1)$, then the definition $\sigma^{t_0}_\alpha$ only depends on $q^\beta_{t_0\downarrow\beta}(\alpha)=q^\beta_{t_1\downarrow\beta}(\alpha)$ and $t_0(\alpha)=t_1(\alpha)$, hence $\sigma^{t_0}_\alpha$ equals $\sigma^{t_1}_\alpha$.
\end{itemize}

At stage $\beta+1$ where $\beta<\omega_1$ is a limit, let $p_0=F_1(\beta)$ if $F_1(\beta)\in P_{\kappa^+}$ and $p_0=1_{P_{\kappa^+}}$ otherwise. Notice that $T_{\beta+1}=T_{\beta}$, since $\delta_\beta\not\in \Upsilon$. For any $t\in T_{\beta+1}$, let $r^\beta_t$ be $p_0\restriction \gamma_t \wedge q^\beta_t$, where $\gamma_t=\sup \{\gamma': p_0\restriction \gamma'\leq q^\beta_t\restriction \gamma'\}$. In particular, if $p_0\leq q^\beta_t$, then $r^\beta_t = p_0$. 

\begin{claim}\label{cohere1}
For $t,t'\in T_{\beta+1}$, if $t\restriction \nu = t'\restriction \nu$, then $r^\beta_{t}\restriction \nu=r^\beta_t\restriction \nu$.
\end{claim}

\begin{proof}[Proof of the claim]
First of all, by the inductive hypothesis, we know $q^\beta_t\restriction \nu = q^\beta_{t'}\restriction \nu$. It is immediate from the definition that either both $\gamma_t, \gamma_{t'}\geq \nu$ or $\gamma_t=\gamma_{t'}<\nu$. In the first case, $r^\beta_t\restriction \nu =p_0\restriction \nu = r^\beta_{t'}\restriction \nu$, and in the second case, $r^\beta_t\restriction \nu = p_0\restriction \gamma_t \wedge q^\beta_t\restriction \nu = p_0\restriction \gamma_{t'} \wedge q^\beta_{t'}\restriction \nu=r^\beta_{t'}\restriction \nu$.
\end{proof}

Next we need to extend $r^\beta_t$ further to meet the genericity requirement. Let $t_\beta\in T_\beta$ be the sequence given by the diamond sequence. We extend $r^\beta_{t_\beta}$ further to some flat condition $q_{t_\beta}^{\beta+1} \in \bigcap_{\gamma\leq \beta, F_1(\gamma) \text{ is a dense open subset of }P_{\kappa^+}} F_1(\gamma)$. For any $t_0,t_1\in T_\beta$, let $\Delta(t_0, t_1)=\sup \{\gamma': t_0\restriction \gamma'=t_1\restriction \gamma'\}$. Given $t'\in T_{\beta+1}$, define $q_{t'}^{\beta+1}=q_{t_\beta}^{\beta+1}\restriction \Delta(t',t_\beta) \wedge r_{t'}^{\beta}$. This is a legitimate condition extending $r^\beta_{t'}$ since $q_{t_\beta}^{\beta+1}\restriction \Delta(t',t_\beta) \leq r^\beta_{t_\beta}\restriction \Delta(t',t_\beta) = r^\beta_{t'}\restriction \Delta(t',t_\beta)$, where the last equality is by Claim \ref{cohere1}. Analogous to the argument in Claim \ref{cohere1}, it is not hard to see that our definition ensures that for $t_0, t_1\in T_{\beta+1}$, if $t_0\restriction \nu = t_1\restriction \nu$, then $q^{\beta+1}_{t_0}\restriction \nu = q^{\beta+1}_{t_1}\restriction \nu$.

\subsection{Building a master condition} Recall that $j^*: V[G]\to M[G*G^*]$ is the lift of $j: V\to M$. By elementarity, in $M[G*G^*]$, $j^*(P_{\kappa^+})$ is an iteration of uniformization forcing with respect to some ladder colorings $\langle \bar{d}^\delta: \delta<j(\kappa^+)\rangle$ on the ladder $j^*(\bar{\eta})$. More precisely, $\bar{d}^\delta$ is a $j^*(P_{\kappa^+})\restriction \delta$-name for a ladder coloring on $j^*(\bar{\eta})$.

\begin{convention}
For $\xi<\omega_1$, $t\in T_\xi$ and $\bar{d}=\langle d^\delta_\kappa: \delta\in j'' A\rangle$ where each $d^\delta_\kappa$ is a (partial) function from $\kappa$ to $2$ and $A\subset (\kappa^+)^{V[G]}$, we say $t$ is \emph{compatible} with $\bar{d}$ if for each $\beta\in \mathrm{dom}(t)\cap A$, $t(\beta)\cup d^{j(\beta)}_\kappa$ is a (partial) function.
\end{convention}

In the following, to simplify notation, $\kappa^+$ will always mean $(\kappa^+)^{V[G]}=(\kappa^+)^V$.
We define a condition $r \in j^*(P_{\kappa^+})$ supported on $j''\kappa^+$ recursively such that it satisfies the following \emph{construction invariant}: for any $\gamma\in j^*(\kappa^+)$, $r\restriction \gamma\Vdash `` j^*(q_t^\xi)\restriction \gamma \in \dot{G}_{j^*(P_{\kappa^+})\restriction \gamma}$ for any $t\in T_\xi $ compatible with $\{d^\delta_\kappa: \delta\in {j}'' \kappa^+ \cap \gamma\}$''.

Suppose $r\restriction j(\beta)$ is already defined satisfying the construction invariant. We define the $j(\beta)$-th component by looking inside the further generic extension by $j^*(P_{\kappa^+})\restriction j(\beta)$ containing $r\restriction j(\beta)$; say $h^*$ is the generic filter.
In $M[G*G^* *h^*]$, we can read off the ladder colorings $\langle \bar{d}^\gamma: \gamma \leq j(\beta)\rangle$. Consider $\bar{e}=\langle d^\delta_{\kappa}: \delta\in j''(\beta+1)\rangle$. Since $j^*(P_{\kappa^+})\restriction j(\beta)$ is $\omega_1$-distributive in $M[G*G^*]$ by Theorem \ref{Shelah}, $\bar{e}$ belongs to $M[G*G^*]$. Now at $j(\beta)$-th coordinate, we take $$\bigcup_{t\in T_{\xi}, t \text{ is compatible with } \bar{e}, q^\xi_t \text{ is flat}, j^*(q^{\xi}_t\restriction \beta)\in h^*} q^{\xi}_t(\beta).$$ 
The definition makes sense since the tree we are building is coherent by requirement (\ref{four}) in the definition of $\{Q^\xi: \xi<\omega_1\}$. There is a slight abuse of notation here since technically speaking, $q^\xi_t(\beta)$ is a $P_\beta$-name instead of a $j^*(P_\beta)$-name. However, the fact that $q^\xi_t$ is flat says that essentially $q^\xi_t(\beta)$ is a function from some initial segment of $\kappa$ to $2$ living in $V[G]$. Hence we can identify $q^\xi_t(\beta)$ with a canonical $j^*(P_\beta)$-name.

\begin{claim}\label{Genericity+Condition}
$r$ is a condition in $j^*(P_{\kappa^+})$ satisfying the construction invariant.
\end{claim}

\begin{proof}
First of all, as $\kappa^+$ is of cardinality $\omega_1$ in $M[G*G^*]$, we know the support of $r$ has the right size. We argue by induction that $r\in j^*(P_{\kappa^+})$ and $r$ satisfies the construction invariant.

At successor stage with $\beta<\kappa^+$ given, suppose $r\restriction j(\beta)\in j^*(P_{\kappa^+})\restriction j(\beta)$ is defined satisfying the construction invariant. We first argue that $r\restriction j(\beta)\Vdash r(j(\beta))$ is a condition in $j^*(\dot{Q}_\beta)$. Let $\dot{\bar{c}}^\beta$ be the ladder coloring on which $\dot{Q}_\beta$ is defined. Let $g\subset j^*(P_\beta)$ be generic over $V[G*G^*]$ containing $r\restriction j(\beta)$. In $V[G*G^**g]$, let $h=\{q_t^\xi\restriction \beta: t\in \bigcup_{\xi<\omega_1}T_\xi,  t \text{ is compatible with }\{d^\delta_\kappa: \delta\in j'' \beta\}\}$. By the construction invariant, we know that ${j^*}'' h \subset g$. Furthermore, it is the case that $h\subset P_{\beta}$ is generic over $V$. The reason is that by $\diamondsuit(\omega_1)$, there exist stationarily many $\xi<\omega_1$ such that $t_\xi\restriction \beta$ is compatible with $\{d^\delta_\kappa: \delta\in j'' \beta\}$. Hence, by the definition of the tree $\{Q^\xi: \xi<\omega_1\}$ (see Subsection \ref{tree}), any dense open subset of $P_\beta$ is met by $h$. 

We can then lift $j^*$ further to $j^+: V[G*h]\to M[G*G^* * g]$. Let $(\dot{\bar{c}}^\beta)^h = \bar{c}^\beta$. By the elementarity, we know $j^+(\bar{c}^\beta)=\bar{d}^{j(\beta)}$. Since $\mathrm{crit}(j^+)=\kappa$, we know that $\langle d^{j(\beta)}_i: i\in \kappa \cap \mathrm{cof}^{M[G*G^* * g]}(\omega_1)\rangle = \bar{c}^\beta$. Suppose $q^{\xi+1}_{t_\xi}$ is given such that $t_\xi$ is compatible with $\langle d^\delta_\kappa: \delta\in j''(\beta+1)\rangle$; then we know that $q^{\xi+1}_{t_\xi}$ is flat, $j^*(q^{\xi+1}_{t_\xi}\restriction \beta)\in h$ and $q^{\xi+1}_{t_\xi}(\beta)$ is indeed a condition in $(j^*(\dot{Q}_\beta))^g$. Note that there are stationarily many $\xi<\omega_1$ satisfying the above. Hence, by the definition of $r$, we know that $r\restriction j(\beta)\Vdash r(j(\beta)) \in j^*(Q_\beta)$. It is also clear from the definition of $r$ that $r\restriction j(\beta)+1$ satisfies the construction invariant.

At limit stages, suppose we are given some limit ordinal $\gamma\in j(\kappa^+)$ so that for any $\zeta<\gamma$, $r\restriction \zeta\in j^*(P_{\kappa^+})\restriction \zeta$ and satisfies the construction invariant. It is clear that $r\restriction \gamma\in j^*(P_{\kappa^+})\restriction \gamma$. Let $r'\leq r\restriction \gamma$ decide $\bar{d}'=\{ d^\delta_\kappa: \delta\in j'' \kappa^+\cap\gamma \}$ in $M[G*G^*]$ by the $\omega_1$-distributivity of $j^*(P_{\kappa^+})\restriction \gamma$ in $M[G*G^*]$. For any $t\in T_\xi$ that is compatible with $\bar{d}'$, we observe that $r'\leq j^*(q^\xi_t)\restriction \gamma$ by the construction invariant up to $\gamma$ and the fact that $j^*(P_{\kappa^+})$ is separative. As a result, $r\restriction \gamma\Vdash `` j^*(q_t^\xi)\restriction \gamma \in \dot{G}_{j^*(P_{\kappa^+})\restriction \gamma}$ for any $t\in \bigcup_{\xi<\omega_1}T_\xi $ that is compatible with $\{d^\delta_\kappa: \delta\in j'' \kappa^+ \cap \gamma\}$''.\end{proof}

Let $r'\leq r$ in $j^*(P_{\kappa^+})$ such that $r'$ decides $\bar{d}=\langle d^\delta_{\kappa}: \delta\in j''\kappa^+\rangle$ to be an element in $M[G*G^*]$ by the $\omega_1$-distributivity of $j^*(P_{\kappa^+})$ in $M[G*G^*]$. Let $H=\langle q^\xi_t: t\in T_\xi, t \text{ is compatible with }\bar{d}\rangle$.
Similar to the argument in Claim \ref{Genericity+Condition}, we show:
\begin{claim} $r'$ is a lower bound for ${j^*}'' H$, and that $H$ meets all dense open subsets of $P_{\kappa^+}$ lying in $V[G]$.
\end{claim}
\begin{proof}
Given $q^\xi_t\in H$, we show by induction on $\beta<j(\kappa^+)$ that $r'\restriction \beta\leq j^*(q^\xi_t)\restriction \beta$. If $\beta\not\in j''\kappa^+$ or $\beta$ is a limit ordinal, then this is immediate. If $\beta=j(\beta')$, we show this for $\beta+1$. By the definition of $r$, and the fact that $r'\restriction \beta$ decides $\bar{d}^*=\langle d^\delta_\kappa: \delta\in j''(\beta'+1)\rangle$, we know that $r'\restriction \beta$ forces $r'(\beta)$ extends $q^\xi_t(\beta')$. To see this, by $\diamondsuit(\omega_1)$, there exist stationarily many $\zeta<\omega_1$ such that $t_\zeta$ is compatible with $\bar{d}^*$. In particular, there exists a limit $\zeta>\xi$ such that $t_\zeta$ is compatible with $\bar{d}^*$. As a result, $q^{\zeta+1}_{t_\zeta}\restriction \beta'+1 \leq q^\xi_t\restriction \beta'+1$. We know $q^{\zeta+1}_{t_\zeta}$ is flat and by the induction hypothesis $r'\restriction \beta \leq  j^*(q^{\zeta+1}_{t_\zeta}\restriction \beta')$. Hence by the definition of $r$, $r'\restriction \beta \Vdash r'(\beta)\leq q^{\zeta+1}_{t_\zeta}(\beta')$. Therefore, $r'\restriction \beta+1 \leq j^*(q^{\zeta+1}_{t_\zeta})\restriction \beta+1 \leq j^*(q^\xi_t)\restriction \beta+1$.

Given a dense open set $D\subset P_{\kappa^+}$, we know that there exists some $\gamma<\omega_1$ such that $F_1(\gamma)=D$. 
By $\diamondsuit(\omega_1)$, we know that there exists stationarily many $\zeta<\omega_1$ such that $t_\zeta$ is compatible with $\bar{d}$. Pick some limit $\zeta>\gamma$ such that $t_\zeta$ is compatible with $\bar{d}$. At stage $\zeta+1$ of the construction, we make sure that $q^{\zeta+1}_{t_\zeta}\in D\cap H$.
\end{proof}

The following claim will finish the proof. In $V[G]$, let $R=\mathrm{Coll}(\omega_1, [\kappa,<j(\kappa)))$ and $\dot{H}$ be a $R$-name for the $H$ defined above.

\begin{claim}\label{countablyclosed}
$R/H=_{def}\{l\in R: \forall p\in P_{\kappa^+}, l\Vdash p\in \dot{H} \rightarrow p\in H, l\Vdash p\not\in \dot{H}\rightarrow p\not\in H\}$ has a dense subset that is countably closed in $V[G*H]$.
\end{claim}

\begin{proof}
Recall $d: \kappa^{++} \leftrightarrow P_{\kappa^+}\cup \{D\subset P_{\kappa^+}: D\text{ is dense open}\}$ is a bijection lying in $V[G]$. Let $R'$ be a dense subset of $R$ consisting of conditions $l$, such that $\mu=_{def} (\kappa^{+++})^V\in dom(l)$ and $l$ decides $F\restriction dom(l(\mu))$, $\langle N_\xi: \xi<dom(l(\mu))\rangle$ and $H\restriction dom(l(\mu))$. We claim that $R'/H$ is countably closed in $V[G*H]$. Suppose not, then let $t\in H$ and a decreasing sequence $\langle p_i\in R'/H: i<\omega\rangle$ be such that $t\Vdash \langle p_i\in R'/H: i<\omega\rangle, p=\bigcup_{i\in \omega} p_i\not\in R'/H$. Extending $t$ if necessary, there must be $p'\in P_{\kappa^+}$ such that $t$ forces $p\Vdash p'\not\in \dot{H}$ but $p'\in \dot{H}$ or $p\Vdash p'\in \dot{H}$ but $p'\not\in \dot{H}$. In either case, we can extend $t$ further if necessary, such that $p\Vdash t\not\in \dot{H}$.

Let $\eta=dom(p(\mu))\in acc (\omega_1)$ and let $p^*=p\cup \{(\mu, \eta, d^{-1}(t))\}$. Since $p$ decides $F\restriction \eta$, $\langle N_\xi: \xi<\eta\rangle$ and $H\restriction \eta$, we already have sufficient data to define $\{q^\xi_s: s\in T_\xi, \xi<\eta\}$, and $H\restriction \eta$ is a branch through the tree. Let $s^*\in T_\eta$ be such that $q^\eta_{s^*}$ is the greatest lower bound of $H\restriction \eta$.

By the construction of the tree at stage $\eta+1$, $p^*$ forces that $F_1(\eta)=t$. As $t$ already forces $H\restriction \eta \subset \dot{H}$, we know that $t$ must extend $q^\eta_{s^*}$. By the construction of the tree, $q^{\eta+1}_{s^*}$ extends $t$. Therefore, $p^*$ must force that $t\in \dot{H}$, which is a contradiction.\end{proof}

\section{Some variations}\label{variations}

\subsection{Variation I: stationary ladder systems}\label{stationaryladder}

In the model from Section \ref{non-stationary}, the ladder system $\bar{\eta}$ witnessing $\mathrm{Unif}_2(\bar{\eta})$ is not stationary. It is a natural question whether we can find a ``larger'' ladder system witnessing the 2-uniformization property.

\begin{definition}\label{T-closed}
Let $T\subset \omega_1$. We say a forcing $R$ is $T$-closed if for all countable $N\prec H(\lambda)$, where $\lambda$ is a sufficiently large regular cardinal, containing $R$ such that $N\cap \omega_1\in T$, and for any $N$-generic decreasing sequence $\langle r_n\in R \cap N : n\in \omega\rangle$, there exists a lower bound $r_\infty\leq r_n$ for all $n\in \omega$.
\end{definition}

Note that if a forcing $R$ is $T$-closed and $T'\subset T\subset \omega_1$, then $R$ is $T'$-closed.

\begin{theorem}\label{stationaryversion}
Relative to the existence of a supercompact cardinal, it is consistent that the following hold: 
\begin{enumerate}
\item GCH,
\item there exists a stationary co-stationary $T\subset \omega_1$ such that
\begin{itemize}
\item $\omega_2$ is generically supercompact via some $T$-closed forcing,  and 
\item there exists a ladder system $\bar{\eta}$ indexed by $T^c$ such that $\mathrm{Unif}_2(\bar{\eta})$ holds.
\end{itemize}
\end{enumerate}
\end{theorem}

\begin{proof}
Since the proof is similar to that in Section \ref{non-stationary}, we only indicate the places requiring modifications. Fix some stationary co-stationary $T\subset acc(\omega_1)$ and assume $\diamondsuit(T)$ holds in the ground model (force $\diamondsuit(T)$ if we need to). We will keep the same notations as in Section \ref{non-stationary}.

Recall the universe we work in is $V[G]$ where $G\subset \mathrm{Coll}(\omega_1, <\kappa)$ is generic over $V$.
Let $\bar{\nu}$ be the nice club ladder system as defined in Section \ref{wellchosen}. The ladder system to force the 2-uniformization property on will be $\bar{\eta}=\bar{\nu}\restriction T^c$. Hence, we can define $P_{\kappa^+}$ as before with respect to $\bar{\eta}$.

In the first stage to construct the tree of generics, the modifications come from relativizing the definition and the construction of $\{Q^{\xi}:\xi<\omega_1\}$ to $T$. More precisely, at (\ref{six}), ``for each limit ordinal $\xi<\omega_1 \cdots$'' is changed to ``for each limit ordinal $\xi\in T \cdots$''. Accordingly, instead of copying a $\diamondsuit(\omega_1)$-sequence, we copy a $\diamondsuit(T)$-sequence. During the construction of $\{Q^\xi: \xi<\omega_1\}$, we do the following: 
\begin{itemize}
\item at limit stages, we take the inverse limit like before;
\item at successor stage $\beta+1$, where $\beta\in T^c$, we do exactly what we did in Section \ref{non-stationary} in the case of successor of successor ordinals;
\item at successor stage $\beta+1$, where $\beta\in T$, we do exactly what we did in Section \ref{non-stationary} in the case of successor of limit ordinals.
\end{itemize}
The rest of the proof carries over.

In the second stage to build a master condition, we need to show: 
\begin{claim}
$R/H=_{def}\{l\in R: \forall p\in P, l\Vdash p\in \dot{H} \rightarrow p\in H, l\Vdash p\not\in \dot{H}\rightarrow p\not\in H\}$ has a dense subset that is $T$-closed in $V[G*H]$.
\end{claim}

\begin{proof}
Recall $d: \kappa^{++} \leftrightarrow P_{\kappa^+}\cup \{D\subset P_{\kappa^+}: D\text{ is dense open}\}$ is a bijection lying in $V[G]$. Let $R'$ be a dense subset of $R$ consisting of conditions $l$, such that $\mu=_{def} (\kappa^{+++})^V\in dom(l)$ and $l$ decides $F\restriction dom(l(\mu))$, $\langle N_\xi: \xi<dom(l(\mu))\rangle$ and $H\restriction dom(l(\mu))$. We claim that $R'/H$ is $T$-closed in $V[G*H]$. 

Suppose not, then let $t\in H$ and a decreasing sequence $\langle p_i\in R'/H: i<\omega\rangle$ be such that $t\Vdash ``\langle p_i\in R'/H: i<\omega\rangle$ is a generic sequence for some $N\prec H(\theta)$ with $N\cap \omega_1\in T$ and $ p=\bigcup_{i\in \omega} p_i\not\in R'/H$''. Note that $dom(p(\mu))\in T$. Extending $t$ further if necessary, we may assume  $p\Vdash t\not\in \dot{H}$.

By the genericity of $\langle p_i: i\in \omega\rangle$ over $N$, we know that $\eta=_{def} dom(p(\mu))\in T$. Let $p^*=p\cup \{(\mu, \eta, d^{-1}(t))\}$. Exactly arguing as in Claim \ref{countablyclosed}, we know that $p^*$ must force that $t\in \dot{H}$, which is a contradiction.\end{proof}
The rest of the proof is the same as in Section \ref{non-stationary}.
\end{proof}

We briefly discuss the strength of the fact that $\omega_2$ is generically supercompact via some $T$-closed forcing for some stationary $T\subset \omega_1$. Many of the following are well-known and we include some proofs for completeness.

\begin{lemma}\label{GenericSupercompactnessTpattern}
If $\omega_2$ is generically supercompact via some $T$-closed forcing where $T\subset acc(\omega_1)$ is stationary, then for any regular cardinal $\lambda\geq \omega_2$, any stationary subset of $\lambda\cap \mathrm{cof}(\omega)$ reflects to an ordinal of cofinality $\omega_1$ with pattern $T$.
\end{lemma}
\begin{proof}
Let a club ladder system $\bar{\nu}=\langle \nu_\delta: \delta\in \lambda\cap \mathrm{cof}(\omega_1)\rangle$ and a stationary $S\subset \lambda\cap \mathrm{cof}(\omega)$ be given. Let $P$ be a $T$-closed forcing such that in $V[G]$ where $G\subset P$ is generic over $V$, there exists an elementary embedding $j: V\to M$ with critical point $\kappa=\omega_2^V$, $j(\kappa)>\lambda$ and $j''\lambda\in M$. Let $\eta=\sup j''\lambda$. Then $M\models \mathrm{cf}(\eta)=\mathrm{cf}(\lambda)=\omega_1$. Let $d=j(\bar{\nu})(\eta)$. Note that $d''T$ is a stationary subset of $\eta$.
We will argue that in $M$, for any club $e\subset \eta$, $j'' S\cap e\cap d''T\neq \emptyset$. With this granted, we can finish the proof as follows: in $M$, $j'' S \cap d'' T$ is a stationary subset of $\eta$, which implies $j(S) \cap d'' T \supset j'' S\cap d'' T$ is a stationary subset of $\eta$. Apply the elementarity of $j$ to get the desired conclusion.
Let the respective $P$-names for the objects defined above be $\dot{j}, \dot{d}, \dot{e}$.

Given $p\in P$, find a countable $N\prec H(\theta)$ for some sufficiently large regular $\theta$ such that $N$ contains $j, \dot{e}, \dot{d}, P, S, \bar{c}$ such that $\delta=N\cap \omega_1\in T$ and $\gamma=\sup N\cap \lambda\in S$.
Let $\bar{p}= \langle p_i: i\in \omega\rangle$ be a decreasing generic sequence for $N$ with $p_0=p$. We observe that: for any $\gamma'<\gamma$, there exists $\delta'<\delta$, $\gamma''\in (\gamma', \gamma)$ and $i\in \omega$ such that $p_i \Vdash ``\dot{d}(\delta')>j(\gamma') \wedge j(\gamma'') >  \min \dot{e} - \dot{d}(\delta')$''. Let $p^*$ be a lower bound for $\langle p_i: i\in \omega\rangle$, then $p^*\Vdash ``\dot{d}(\delta)=\sup_{\delta'<\delta} \dot{d}(\delta')=\sup_{\gamma''<\gamma} j(\gamma'')=j(\gamma)\in \dot{e}$'', noting that $j$ is forced to be continuous at ordinals of countable cofinality and $\dot{d}$ is forced to be continuous. Therefore, $p^*\Vdash j(\gamma)\in j''S\cap \dot{e}\cap \dot{d}''T$.\end{proof}

\begin{definition}
For any cardinal $\lambda\geq \omega_2$, we say the \emph{Weak Reflection Principle} holds at $\lambda$, or $\mathrm{WRP}(\lambda)$ holds, if any stationary subset $S\subset [\lambda]^\omega$ \emph{reflects}, namely, there exists $W\in [\lambda]^{\aleph_1}$ containing $\omega_1$ such that $S\cap [W]^{\omega}$ is stationary in $[W]^\omega$. Given a stationary $S\subset [\lambda]^\omega$, we use $\mathrm{WRP}(S)$ to abbreviate the assertion that any stationary subset of $S$ reflects.
\end{definition}

\begin{remark}
$\mathrm{WRP}$, namely that ``$\mathrm{WRP}(\lambda)$ holds for all $\lambda$'', has high consistency strength. For example, by \cite[Theorem 3.8]{MR1174395}, it implies $\square(\lambda)$ fails for all regular cardinal $\lambda \geq \omega_2$. By \cite{MR2499432}, this in turn implies the existence of an inner model with a proper class of strong cardinals and a proper class of Woodin cardinals. 
\end{remark}

The following are well-known and proofs can be found in \cite{MR924672}. Let $T\subset \omega_1$ be a stationary set.

\begin{theorem}[\cite{MR924672}]
If $\omega_2$ is generically supercompact via some $T$-closed forcing, then for any cardinal $\lambda\geq \omega_2$, $\mathrm{WRP}(\{x\in [\lambda]^\omega: x\cap \omega_1\in T\})$ holds. In particular, $\mathrm{NS}_{\omega_1}\restriction T$ is presaturated.
\end{theorem}

To summarize, relative to the existence of large cardinals, it is consistent with GCH that there exists a stationary co-stationary set $T\subset \omega_1$ such that $\mathrm{Refl}_T(S^2_0)$, $\mathrm{NS}_{\omega_1}\restriction T$ is presaturated, and there exists a ladder system $\bar{\eta}$ indexed by $T^c$ such that $\mathrm{Unif}_2(\bar{\eta})$ holds, which implies $\neg \diamondsuit(S^2_1)$. This provides a contrast to Theorem \ref{ShelahStarting}.

\subsection{Variation II: constant ladder colorings}

As we remarked before, a theorem of Shelah \cite{MR1623206} asserts it is impossible to get the 2-uniformization property on a club ladder system. However, if we only restrict to constant colorings, then it is indeed possible.

\begin{definition}
Given a ladder system $\bar{\eta}$ on $S^2_1$, we say the \emph{m-uniformization property} holds for $\bar{\eta}$ (abbreviated as $\mathrm{Unif}_m(\bar{\eta})$), if for any \emph{constant} ladder coloring $\bar{c}$ on $\bar{\eta}$, there exists a uniformizing function for $\bar{c}$.
\end{definition}

The ``m'' above stands for ``monochromatic''. Note that if there exists some ladder $\bar{\eta}$ on $S^2_1$ such that $\mathrm{Unif}_m(\bar{\eta})$ holds, then $\neg \diamondsuit(S^2_1)$.

Almost the same proof as that from Theorem \ref{Shelah} will give (the modifications can also be read off from the proof of Theorem \ref{constant}): 
\begin{theorem}[Shelah]
Suppose $\bar{\eta}$ is a nice club ladder system and $\diamondsuit(\omega_1)$ holds. Let $\langle P_\gamma, \dot{Q}_\beta: \gamma\leq \kappa^+, \beta<\kappa^+\rangle$ be a $<\kappa$-support iteration of m-uniformization forcings with respect to $\bar{\eta}$.
Then the following hold: 
\begin{enumerate}
\item the set of flat conditions is dense in $P_{\kappa^+}$ and
\item $P_{\kappa^+}$ is $\omega_1$-distributive.
\end{enumerate}
\end{theorem}

\begin{theorem}\label{constant}
Relative to the existence of a supercompact cardinal, it is consistent that the following hold: 
\begin{enumerate}
\item GCH,
\item $\omega_2$ is generically supercompact via some countably closed forcing and
\item $\mathrm{Unif}_m(\bar{\eta})$ holds for some club ladder system $\bar{\eta}$ on $S^2_1$.
\end{enumerate}
\end{theorem}
\begin{proof}
Since the proof is similar to that in Section \ref{non-stationary}, we only indicate the places requiring modifications. We keep in the same notations as in Section \ref{non-stationary}. The ladder system we will force the m-uniformization property on is $\bar{\eta}=\bar{\nu}$. We may assume $\diamondsuit(\omega_1)$ holds in the ground model.

In the first stage to construct the tree of generics, the first modification comes from the definition of $T_\xi$. Here $T_\xi$ only consists of those $t$ such that for any $\beta\in dom(t)$, $t(\beta)$ is constant on its domain.
The second modification comes from the construction at the successor of limit stages. At stage $\beta+1$ where $\beta$ is a limit ordinal, let $p_0=F_1(\beta)$ if $F_1(\beta)\in P_{\kappa^+}$ and $p_0=1_{P_{\kappa^+}}$ otherwise. Given $t\in T_{\beta+1}$, we need to define $q^{\beta+1}_t$. Note that it is no longer the case that $T_\beta=T_{\beta+1}$. However, the \emph{key point} is that any element $l\in T_\beta$ has a unique extension $l^+$ into $T_{\beta+1}$. The key reason is that if $b\in T_{\beta+1}$ satisfies that $b\downarrow \beta =l$, then for any $\zeta<\beta$, $b(F(\zeta))(\delta_\beta)$ is completely determined by $b(F(\zeta))\restriction \delta_\beta = l(F(\zeta))$, since we only deal with constant colorings.

 Let $t'=t\downarrow \beta$. We first define $r_{t}^\beta$. We first extend $q_{t'}^\beta$ to $a_{t}^\beta$ canonically obeying $t$ inside $N_{\beta+1}$, then we define $r_{t}^\beta$ to be $p_0\restriction \gamma_t \wedge a_{t}^\beta$ where $\gamma_t = \sup \{\gamma: p_0\restriction \gamma\leq a_{t}^\beta\restriction \gamma\}$. Exactly as before, we can show: 
\begin{claim}\label{cohere2}
For $t,t'\in T_{\beta+1}$, if $t\restriction \nu = t'\restriction \nu$, then $r^\beta_{t}\restriction \nu=r^\beta_{t'}\restriction \nu$.
\end{claim}

The way to define $q^{\beta+1}_t$ from $r^\beta_t$ is the same as that in Section \ref{tree}.

In the second stage to build a master condition, we need to show: 
\begin{claim}
$R/H=_{def}\{l\in R: \forall p\in P, l\Vdash p\in \dot{H} \rightarrow p\in H, l\Vdash p\not\in \dot{H}\rightarrow p\not\in H\}$ has a dense subset that is countably closed in $V[G*H]$.
\end{claim}

\begin{proof}
Recall $d: \kappa^{++} \leftrightarrow P_{\kappa^+}\cup \{D\subset P_{\kappa^+}: D\text{ is dense open}\}$ is a bijection lying in $V[G]$. Let $R'$ be a dense subset of $R$ consisting of conditions $l$, such that $\mu=_{def} (\kappa^{+++})^V\in dom(l)$ and $l$ decides $F\restriction dom(l(\mu))$, $\langle N_\xi: \xi<dom(l(\mu))\rangle$ and $H\restriction dom(l(\mu))$. We claim that $R'/H$ is countably closed in $V[G*H]$. Suppose not, then let $t\in H$ and $\langle p_i\in R'/H: i<\omega\rangle$ be such that $t\Vdash \langle p_i\in R'/H: i<\omega\rangle, p=\bigcup_{i\in \omega} p_i\not\in R'/H$. Extending $t$ in $H$ if necessary, we may assume $p\Vdash t\not\in \dot{H}$.

Let $\eta=dom(p(\mu))\in \lim (\omega_1)$. Since $p$ decides $F\restriction \eta$, $\langle N_\xi: \xi<\eta\rangle$ and $H\restriction \eta$, we already have sufficient data to define $\{q^\xi_s: s\in T_\xi, \xi<\eta\}$, and also $H\restriction \eta$ is a branch through the tree. Let $q^\eta_{s^*}$ be the greatest lower bound for $H\restriction \eta$ for some $s^*\in T_\eta$. As $t$ already forces $H\restriction \eta \subset \dot{H}$, we know that $t$ extends $q^\eta_{s^*}$. Let $s^+$ be the unique extension of $s^*$ into $T_{\eta+1}$. Since $t\neq q^\eta_{s^*}$ as $p\Vdash q^\eta_{s^*}\in \dot{H}$ and $t\in H$, $t$ must obey $s^+$.

Let $p^*=p\cup \{(\mu, \eta, d^{-1}(t))\}$. We know that $p^*$ forces that $F_1(\eta)=t$. By the construction of the tree at stage $\eta+1$, any node above $q^\eta_{s^*}$ will extend $t$. Therefore, $p^*$ must force that $t\in \dot{H}$, which is a contradiction.\end{proof}
The rest of the proof is the same as in Section \ref{non-stationary}.\end{proof}

\section{From a cheaper assumption}\label{cheaper}

If our goal is mainly GCH, $\neg \diamondsuit(S^2_1)$ and some degree of stationary reflection at $\omega_2$, without worrying about the saturation properties of $\mathrm{NS}_{\omega_1}$, then we can get the model from a much cheaper assumption.

\begin{theorem}\label{weakcompactness}
Relative to the existence of a weakly compact cardinal, it is consistent that 
\begin{enumerate}
\item $\mathrm{GCH}$ holds,
\item $\mathrm{WRP}(\omega_2)$ and
\item there exists a ladder system $\bar{\eta}$ on $S^2_1$ such that $\mathrm{Unif}_2(\bar{\eta})$ holds.
\end{enumerate}
\end{theorem}

We need the following characterization of weakly compact cardinals due to Hauser.

\begin{theorem}[Hauser \cite{MR1166462}]\label{Hauser}
The following are equivalent for an inaccessible cardinal $\kappa$: 
\begin{enumerate}
\item $\kappa$ is weakly compact,
\item for any transitive set $M$ of size $\kappa$ with ${}^{<\kappa}M\subset M$ and $\kappa\in M$, there exists an elementary embedding $j: M\to N$ where $N$ is transitive of size $\kappa$, ${}^{<\kappa}N\subset N$, $\mathrm{crit}(j)=\kappa$ and $j, M\in N$.
\end{enumerate}
\end{theorem}

\begin{proof}[Proof of Theorem \ref{weakcompactness}]

Let $\kappa$ be a weakly compact cardinal. We may assume the ground model satisfies $\diamondsuit(\omega_1)$ since if not, we can always force it.
The forcing will be the same as the forcing used in the proof of Theorem \ref{main} (see Section \ref{non-stationary}). In particular, the ladder system to force the 2-uniformization property on is also the same as the one in Section \ref{non-stationary}. The proof is similar so we will only sketch some main points. We will keep the same notations as in Section \ref{non-stationary}.

Suppose for the sake of contradiction that there exists $(p,\dot{q})\in \mathrm{Coll}(\omega_1, <\kappa)*P_{\kappa^+}$ such that $(p,\dot{q})\Vdash ``$ there exists a stationary $S\subset [\omega_2]^\omega$ that does not reflect''. Let $\dot{S}$ be a $\mathrm{Coll}(\omega_1, <\kappa)*P_{\kappa^+}$-name for the set $S$ above. Without loss of generality, we may assume that $\dot{S}\in H(\kappa^+)$. The reason is that $\mathrm{Coll}(\omega_1, <\kappa)*P_{\kappa^+} \subset H(\kappa^+)$ and $\mathrm{Coll}(\omega_1, <\kappa)*P_{\kappa^+}$ is $\kappa^+$-c.c.

Back in $V$, find $X\prec H(\theta')$, where $\theta'$ is a sufficiently large regular cardinal, containing $\mathrm{Coll}(\omega_1, <\kappa)*P_{\kappa^+}$, $\dot{S}$, $(p,\dot{q})$, and ${}^{<\kappa}X\subset X$ of size $\kappa$. Let $M$ be the transitive collapse of $X$. Note that $\dot{S}$ and $(p,\dot{q})$ are fixed by the transitive collapse. Let $\mathrm{Coll}(\omega_1, <\kappa)*\bar{P}$ be the image of $\mathrm{Coll}(\omega_1, <\kappa)*P_{\kappa^+}$ under the transitive collapse of $X$. Let $\delta=X\cap \kappa^+$.
 It is not hard to see that $\mathrm{Coll}(\omega_1, <\kappa)*\bar{P}=\mathrm{Coll}(\omega_1, <\kappa)* P_{\delta}$.
Apply Theorem \ref{Hauser} to get elementary $j: M\to N$ such that $N$ is transitive of size $\kappa$, ${}^{<\kappa}N\subset N$, $\mathrm{crit}(j)=\kappa$ and $j, M\in N$.

Notice that 
\begin{itemize}
\item $M\models (p,\dot{q})\Vdash_{\mathrm{Coll(\omega_1, <\kappa)}*P_{\delta}} \dot{S}$ is non-reflecting. This just follows from the elementarity of $X$ and the fact that the transitive collapse map from $X$ to $M$ fixes $(p,\dot{q})$ and $\dot{S}$. 
\item $N\models (p,\dot{q})\Vdash_{\mathrm{Coll(\omega_1, <\kappa)}*P_{\delta}} \dot{S}$ is stationary. Otherwise, there exists $$(p',\dot{q}')\leq_{\mathrm{Coll(\omega_1, <\kappa)}*P_{\delta}} (p,\dot{q})$$ forcing over $N$ that $\dot{S}$ is not stationary. Take $W\subset\mathrm{Coll(\omega_1, <\kappa)}*P_{\kappa^+}$ generic over $V$ containing $(p',\dot{q}')$, then in $N[W]$, $(\dot{S})^W$ is not stationary. But in $V[W]\supset N[W]$, $(\dot{S})^W$ is stationary by the assumption, which is impossible.

\end{itemize}

Let $G\subset \mathrm{Coll}(\omega_1, <\kappa)$ containing $p$ be generic over $V$ and $G^*\subset \mathrm{Coll}(\omega_1, [\kappa, <j(\kappa)))$ be generic over $V[G]$ (the forcing is equivalent to $\mathrm{Coll}(\omega_1, \kappa)$ from the perspective of $V[G]$).

We can lift $j$ to $j^*: M[G]\to N[G*G^*]$. 
Our goal is to build in $N[G*G^*]$ a generic $H\subset P_\delta$ containing $q=(\dot{q})^G$ over $N[G]$, such that in $N[G*H]$, $\mathrm{Coll}(\omega_1, [\kappa, <j(\kappa)))/H$ is countably closed and ${j^*}'' H$ has a lower bound $r$ in $j^*(P_\delta)$. Suppose this can be accomplished. Let $h\subset j^*(P_\delta)$ containing $r$ be generic over $N[G*G^*]$. Then we can lift $j^*$ further to $j^+: M[G*H]\to N[G*G^**h]$. Let $S=(\dot{S})^{G*H}$. Then $S\in M[G*H] \subset N[G*H]$. Furthermore, we know that $S$ is stationary in $N[G*H]$. Since the forcing $\mathrm{Coll}(\omega_1, [\kappa, <j(\kappa)))/H$ is countably closed in $N[G*H]$, $S$ remains stationary in $N[G*G^*]$. Since $j^*(P_\delta)$ is countably closed in $N[G*G^*]$, $S$ remains stationary in $N[G*G^**h]$. Therefore, $j^+(S)\cap [\kappa]^\omega \supseteq S$ is stationary in $N[G*G^**h]$. By the elementarity of $j^+$, we conclude that $S$ reflects in $M[G*H]$, which contradicts with our assumption that $S$ is nonreflecting in $M[G*H]$.

To build such an $H$ satisfying the requirements as above, we basically follow the same argument from Section \ref{non-stationary} in $N[G*G^*]$. We will not repeat the argument but the salient points are: 
\begin{itemize}
\item $N\models \diamondsuit(\omega_1)$,
\item $j(\kappa)$ is a strongly inaccessible cardinal in $N[G]$, 
\item $P_\delta$ has size $\kappa$ in $N[G]$ since $M\in N$, which implies the number of dense open subsets of $P_\delta$ in $N[G]$ is $<j(\kappa)$ and 
\item $j'' \delta \in N$.   \qedhere
\end{itemize}\end{proof}

\begin{remark}
By adapting the forcing variations in Section \ref{variations} to the level of a weakly compact cardinal, we can get: relative to the existence of a weakly compact cardinal,
\begin{enumerate}
\item it is consistent that there exists a stationary co-stationary $T\subset \omega_1$ such that 
\begin{enumerate}
\item $\mathrm{GCH}$ holds,
\item $\mathrm{WRP}(\{x\in [\omega_2]^\omega: x\cap \omega_1\in T\})$ and
\item there exists a ladder system $\bar{\eta}$ indexed by $T^c$ such that $\mathrm{Unif}_2(\bar{\eta})$ holds.
\end{enumerate}

\item it is consistent that the following hold: 
\begin{enumerate}
\item GCH,
\item $\mathrm{WRP}(\omega_2)$ and
\item $\mathrm{Unif}_m(\bar{\eta})$ holds for some club ladder system $\bar{\eta}$ on $S^2_1$.
\end{enumerate}

\end{enumerate} 

\end{remark}

\begin{remark}
In the model constructed in the proof of Theorem \ref{weakcompactness}, we have GCH, $\mathrm{Refl}(S^2_0)$, $\neg \diamondsuit(S^2_1)$ and there does not exist a $\square(\omega_2, \omega_1)$-sequence $\langle \mathcal{C}_\alpha: \alpha<\omega_2\rangle$ such that $\{\alpha\in S^2_0: |\mathcal{C}_\alpha|\leq \aleph_0\}$ is stationary. The last item can be seen to hold in that model by a similar argument to the one given showing that $\mathrm{WRP}(\omega_2)$ holds. It is also well-known that $\mathrm{WRP}(\omega_2)$ implies $\mathrm{Refl}(S^2_0)$.

Any model satisfying GCH and the $\aleph_2$-Suslin hypothesis (namely, there is no $\aleph_2$-Suslin tree) will need to satisfy the configuration above, see for example \cite{MR309729}, \cite{MR4019621}, \cite{MR485361}, \cite{MR3628222}.
\end{remark}

\section{Some questions and remarks}\label{questions}

\begin{question}
Is it possible to produce a model satisfying
\begin{enumerate}
\item GCH,
\item $\mathrm{Refl}(S^2_0)$ and 
\item there exists a ladder system $\bar{\eta}$ on $S^2_1$ such that $\mathrm{Unif}_2(\bar{\eta})$ holds
\end{enumerate}
from the existence of a Mahlo cardinal?
\end{question}

A rough idea will be to mix the iterations to kill non-reflecting stationary subsets with the iterations to add witnesses to the uniformization property. The difference is that the first forcing is $\omega_1$-\emph{descendingly complete}, in the sense that any $\omega_1$-length decreasing sequence of conditions will have a lower bound. But it is in general not countably closed. While the situation for the second forcing is the other way round.

The next question arises from the technicality in the proof of Theorem \ref{stationaryversion}, where one fixes a stationary co-stationary set first, managing the generic large cardinal property on the set and the uniformization property on its complement. One may wonder if there is a ZFC reason.

\begin{question}\label{solved}
Is the following jointly consistent (relative to appropriate large cardinal assumptions)?
\begin{enumerate}
\item GCH,
\item $\omega_2$ is generically supercompact via some countably closed forcing, and
\item there exists a ladder system $\bar{\eta}$ on $S^2_1$ indexed by some stationary subset $T\subset \omega_1$ such that $\mathrm{Unif}_2(\bar{\eta})$ holds
\end{enumerate}
\end{question}

Assaf Rinot pointed out that the hypothesis of Question \ref{solved} is not consistent. More specifically, he showed that if there exists a ladder system $\bar{\eta}$ on $S^2_1$ indexed by some stationary subset $T\subset \omega_1$ such that $\mathrm{Unif}_2(\bar{\eta})$ holds, then $\neg \mathrm{Refl}_T(S^2_0)$ (see Lemma \ref{GenericSupercompactnessTpattern}). To see this, note that $\mathrm{Unif}_2(\bar{\eta})$ implies $\mathrm{Unif}_{\aleph_1}(\bar{\eta})$ (see Remark \ref{UnifImpliesCH}). Consider the ladder coloring $\bar{c}=\langle c_\delta: \delta\in S^2_1\rangle$ such that for any $\delta\in S^2_1$, $c_\delta$ is injective. Suppose $g: \omega_2 \to \omega_1$ is the uniformizing function. Let $\alpha<\omega_1$ be such that $S=g^{-1}(\alpha)\cap S^2_0$ is stationary. Then it is easy to see that $S$ witnesses $\neg \mathrm{Refl}_T(S^2_0)$. As a result, Theorem \ref{stationaryversion} is optimal in a sense.

There are many other compactness principles at $\omega_2$ whose relationship with $\diamondsuit(S^2_1)$ under $2^{\omega_1}=\omega_2$ can be investigated. For example, a theorem of Baumgartner \cite{MR776640} (see also the end of page 68 in \cite{MR1359154}) states that:
%
If $2^{\omega_1}=\omega_2$ and there does not exist a \emph{weak Kurepa tree}, namely a $\aleph_1$-sized tree of height $\omega_1$ with $>\aleph_1$ many branches, 
then $\diamondsuit^*(S^2_1)$ holds.\textsuperscript{1}\footnotetext{\textsuperscript{1}I would like to thank the referee for pointing this out.}

%

The nonexistence of a weak Kurepa tree necessarily implies the failure of CH.
Known models of the hypothesis include any model of PFA and the Mitchell model for the tree property at $\aleph_2$ from \cite{MR0313057}. But it is also known that the tree property at $\aleph_2$ is consistent with the existence of a Kurepa tree \cite{MR3749402}.

Recall Remark \ref{UnifImpliesCH} that if there exists a ladder system $\bar{\eta}$ on $S^2_1$ such that $\mathrm{Unif}_2(\bar{\eta})$ holds, then CH holds. On the other hand, $\neg \diamondsuit(S^2_1)$ is compatible with the continuum being arbitrarily large, since no c.c.c forcing can add a $\diamondsuit(S^2_1)$-sequence over a model of $\neg \diamondsuit(S^2_1)$.


\begin{question}
Is it consistent that the tree property holds at $\aleph_2$, $2^{\omega_1}=\omega_2$ and $\neg \diamondsuit(S^2_1)$?
\end{question}

\section{Acknowledgment}

I want to thank Menachem Magidor, Assaf Rinot and Spencer Unger for stimulating discussions, comments and suggestions.
I am grateful to the referee for an extremely careful reading and a comprehensive list of suggestions and corrections.

\bibliographystyle{alpha}
\bibliography{bib}

\begin{thebibliography}{BNGH19}

\bibitem[Bau84]{MR776640}
James~E. Baumgartner.
\newblock Applications of the proper forcing axiom.
\newblock In {\em Handbook of set-theoretic topology}, pages 913--959.
  North-Holland, Amsterdam, 1984.

\bibitem[BN19]{MR3960897}
Omer Ben-Neria.
\newblock Diamonds, compactness, and measure sequences.
\newblock {\em J. Math. Log.}, 19(1):1950002, 20, 2019.

\bibitem[BNGH19]{MR3914938}
Omer Ben-Neria, Shimon Garti, and Yair Hayut.
\newblock Weak prediction principles.
\newblock {\em Fund. Math.}, 245(2):109--125, 2019.

\bibitem[Cum]{CummingsWoodin}
J.~Cummings.
\newblock Woodin's theorem on killing diamond via radin forcing.
\newblock {\em Unpublished note}.

\bibitem[Cum18]{MR3749402}
James Cummings.
\newblock Aronszajn and {K}urepa trees.
\newblock {\em Arch. Math. Logic}, 57(1-2):83--90, 2018.

\bibitem[DH06]{MR2279655}
Mirna D\v{z}amonja and Joel~David Hamkins.
\newblock Diamond (on the regulars) can fail at any strongly unfoldable
  cardinal.
\newblock {\em Ann. Pure Appl. Logic}, 144(1-3):83--95, 2006.

\bibitem[FM95]{MR1359154}
Matthew Foreman and Menachem Magidor.
\newblock Large cardinals and definable counterexamples to the continuum
  hypothesis.
\newblock {\em Ann. Pure Appl. Logic}, 76(1):47--97, 1995.

\bibitem[FMR20]{fakereflection}
G~Fernanades, M~Moreno, and A~Rinot.
\newblock Fake reflection.
\newblock {\em preprint}, 2020.

\bibitem[FMS88]{MR924672}
M.~Foreman, M.~Magidor, and S.~Shelah.
\newblock Martin's maximum, saturated ideals, and nonregular ultrafilters. {I}.
\newblock {\em Ann. of Math. (2)}, 127(1):1--47, 1988.

\bibitem[Gol16]{golshani2016weak}
Mohammad Golshani.
\newblock (weak) diamond can fail at the least inaccessible cardinal, 2016.

\bibitem[GR12]{MR2869191}
Moti Gitik and Assaf Rinot.
\newblock The failure of diamond on a reflecting stationary set.
\newblock {\em Trans. Amer. Math. Soc.}, 364(4):1771--1795, 2012.

\bibitem[Gre76]{MR485361}
John Gregory.
\newblock Higher {S}ouslin trees and the generalized continuum hypothesis.
\newblock {\em J. Symbolic Logic}, 41(3):663--671, 1976.

\bibitem[Ham02]{Hamkins:LaverDiamond}
Joel~David Hamkins.
\newblock A class of strong diamond principles.
\newblock {\em ArXiv e-prints}, 2002.

\bibitem[Hau92]{MR1166462}
Kai Hauser.
\newblock The indescribability of the order of the indescribable cardinals.
\newblock {\em Ann. Pure Appl. Logic}, 57(1):45--91, 1992.

\bibitem[Jen72]{MR309729}
R.~Bj\"{o}rn Jensen.
\newblock The fine structure of the constructible hierarchy.
\newblock {\em Ann. Math. Logic}, 4:229--308; erratum, ibid. 4 (1972), 443,
  1972.
\newblock With a section by Jack Silver.

\bibitem[JK69]{JensenKunen}
Ronald Jensen and Kenneth Kunen.
\newblock Some combinatorial properties of {L} and {V}.
\newblock {\em Unpublished manuscript}, 1969.

\bibitem[JSSS09]{MR2499432}
Ronald Jensen, Ernest Schimmerling, Ralf Schindler, and John Steel.
\newblock Stacking mice.
\newblock {\em J. Symbolic Logic}, 74(1):315--335, 2009.

\bibitem[K\"07]{MR2298486}
Bernhard K\"{o}nig.
\newblock Forcing indestructibility of set-theoretic axioms.
\newblock {\em J. Symbolic Logic}, 72(1):349--360, 2007.

\bibitem[Lav78]{MR0472529}
Richard Laver.
\newblock Making the supercompactness of {$\kappa $} indestructible under
  {$\kappa $}-directed closed forcing.
\newblock {\em Israel J. Math.}, 29(4):385--388, 1978.

\bibitem[Mag82]{MR683153}
Menachem Magidor.
\newblock Reflecting stationary sets.
\newblock {\em J. Symbolic Logic}, 47(4):755--771 (1983), 1982.

\bibitem[Mit73]{MR0313057}
William Mitchell.
\newblock Aronszajn trees and the independence of the transfer property.
\newblock {\em Ann. Math. Logic}, 5:21--46, 1972/73.

\bibitem[Rin10]{MR2723781}
Assaf Rinot.
\newblock A relative of the approachability ideal, diamond and non-saturation.
\newblock {\em J. Symbolic Logic}, 75(3):1035--1065, 2010.

\bibitem[Rin11]{MR2777747}
Assaf Rinot.
\newblock Jensen's diamond principle and its relatives.
\newblock In {\em Set theory and its applications}, volume 533 of {\em Contemp.
  Math.}, pages 125--156. Amer. Math. Soc., Providence, RI, 2011.

\bibitem[Rin17]{MR3628222}
Assaf Rinot.
\newblock Higher {S}ouslin trees and the {GCH}, revisited.
\newblock {\em Adv. Math.}, 311:510--531, 2017.

\bibitem[Rin19]{MR4019621}
Assaf Rinot.
\newblock Souslin trees at successors of regular cardinals.
\newblock {\em MLQ Math. Log. Q.}, 65(2):200--204, 2019.

\bibitem[She84]{ShelahStrange}
S.~Shelah.
\newblock An $\aleph _2$ souslin tree from a strange hypothesis.
\newblock {\em Abstracts American Math Soc 160 (1984) 198}, 1984.

\bibitem[She98]{MR1623206}
Saharon Shelah.
\newblock {\em Proper and improper forcing}.
\newblock Perspectives in Mathematical Logic. Springer-Verlag, Berlin, second
  edition, 1998.

\bibitem[She10]{MR2596054}
Saharon Shelah.
\newblock Diamonds.
\newblock {\em Proc. Amer. Math. Soc.}, 138(6):2151--2161, 2010.

\bibitem[SK80]{MR597452}
Charles~I. Steinhorn and James~H. King.
\newblock The uniformization property for {$\aleph _{2}$}.
\newblock {\em Israel J. Math.}, 36(3-4):248--256, 1980.

\bibitem[Tod82]{MR656600}
Stevo Todor\v{c}evi\'{c}.
\newblock Some combinatorial properties of trees.
\newblock {\em Bull. London Math. Soc.}, 14(3):213--217, 1982.

\bibitem[Vel92]{MR1174395}
Boban Veli\v{c}kovi\'{c}.
\newblock Forcing axioms and stationary sets.
\newblock {\em Adv. Math.}, 94(2):256--284, 1992.

\bibitem[Zem10]{MR2587470}
Martin Zeman.
\newblock Diamond, {GCH} and weak square.
\newblock {\em Proc. Amer. Math. Soc.}, 138(5):1853--1859, 2010.

\end{thebibliography}

\Addresses

\end{document}